\documentclass[reqno]{amsart}

\usepackage{graphicx,enumerate,booktabs,nicefrac,color}
\usepackage[foot]{amsaddr}

\newcommand{\norm}[1]{\left\|#1\right\|}
\newcommand{\R}{\mathbb{R}}
\newcommand{\N}{\mathbb{N}}
\newcommand{\C}{\mathbb{C}}
\newcommand{\abs}[1]{\left|#1\right|}
\newcommand{\F}{\mathsf{F}}
\newcommand{\NF}{\mathsf{N}}
\renewcommand{\d}{\mathsf{d}}
\newcommand{\ds}{\,\d s}
\newcommand{\dx}{\,\d x}

\newtheorem{theorem}{Theorem}[section]

\theoremstyle{definition}

\newtheorem{example}[theorem]{Example}

\newtheorem{algorithm}[theorem]{Algorithm}
\newtheorem{remark}[theorem]{Remark}

\title[Adaptive Newton Methods]{An Adaptive Newton-Method\\ Based on a Dynamical Systems Approach}

\author[M.~Amrein]{Mario Amrein}

\author[T.~P.~Wihler]{Thomas P.~Wihler}
\address{Mathematics Institute, University of Bern, CH-3012 Switzerland}
\email{wihler@math.unibe.ch}

\begin{document}

\begin{abstract}
The traditional Newton method for solving nonlinear operator equations in Banach spaces is discussed within the context of the continuous Newton method. This setting makes it possible to interpret the Newton method as a discrete dynamical system and thereby to cast it in the framework of an adaptive step size control procedure. In so doing, our goal is to reduce the chaotic behavior of the original method without losing its quadratic convergence property close to the roots. The performance of the modified scheme is illustrated with various examples from algebraic and differential equations.
\end{abstract}

\keywords{Newton-Raphson methods, continuous Newton-Raphson method, adaptive step size control, nonlinear differential equations, chaotic behavior.}

\subjclass[2010]{49M15,58C15,37D45,74H65}

\maketitle

\section{Introduction}

Let $ X, Y $ be two Banach spaces, with norms~$\|\cdot\|_X$ and~$\|\cdot\|_Y$, respectively. Given an open subset~$\Omega\subset X$, and a continuous (possibly nonlinear) operator~$\F:\,\Omega\to Y$, we are interested in finding the zeros~$x\in\Omega$ of~$\F$, i.e., we aim to solve the operator equation
\begin{equation}\label{eq:F0}
x\in\Omega:\qquad\F(x)=0.
\end{equation}
Supposing that the Fr\'echet derivative~$\F'$ of~$\F$ exists in~$\Omega$ (or in a suitable subset), the classical Newton-Raphson method for solving~\eqref{eq:F0} starts from an initial guess~$x_0\in\Omega$, and generates the iterates
\begin{equation}\label{eq:newton}
x_{n+1}=x_n+\delta_n,
\end{equation}
where the update~$\delta_n\in X$ is implicitly given by the {\em linear} equation
\[
\F'(x_n)\delta_n=-\F(x_n),
\]
for~$n\ge 0$. Naturally, we need to assume that~$\F'(x_n)$ is invertible for all~$n\ge 0$, and that~$\{x_n\}_{n\ge 0}\subset\Omega$. 

Newton's method features both local as well as global properties. On the one hand, on a {\em local} level, the scheme is often celebrated for its quadratic convergence regime `sufficiently close' to a root. From a {\em global} perspective, on the other hand, the Newton method is well-known to exhibit chaotic behavior. Indeed, the original works of Fatou~\cite{fatou} and Julia~\cite{julia}, for instance, revealed that applying the Newton method to algebraic systems of equations may result in highly complex or even fractal attractor boundaries of the associated roots. This was confirmed in the 1980s when computer graphics were employed to illustrate the theoretical results numerically; see, e.g., \cite{peitgen}.

In order to tame the chaotic behavior of Newton's method a number of different ideas have been proposed in the literature. In particular, the use of damping aiming to avoid the appearance of possibly large updates in the iterations, constitutes a popular approach in practical applications. More precisely, \eqref{eq:newton} is replaced with
\[
x_{n+1}=x_n+\alpha\delta_n,
\]
for a possibly small damping parameter~$0<\alpha<1$. More sophisticatedly, variable damping may lead to more efficient results;  see, e.g., the extensive overview~\cite{5} or \cite{DSB95,epureanu:102,Varona02} for different variations of the classical Newton scheme. The idea of adaptively adjusting the magnitude of the Newton updates has also been studied in the recent article~\cite{ScWi11}; there, following, e.g.,~\cite{neuberger,peitgen,smale}, the Newton method was identified as the numerical discretization of a specific ordinary differential equation (ODE)---the so-called continuous Newton method---by the explicit Euler scheme, with a fixed step size~$h=1$. Then, in order to tame the chaotic behavior of the Newton iterations, the idea presented in~\cite{ScWi11} is based on discretizing the continuous Newton ODE by the explicit Euler method with variable step sizes, and to combine it with a simple step size control procedure; in particular, the resulting procedure retains the optimal step size~1 whenever sensible and is able to deal with singularities in the iterations more carefully than the classical Newton scheme. In fact, numerical experiments revealed that the new method is able to generate attractors with almost smooth boundaries where the traditional Newton method produces fractal Julia sets. Moreover, the numerical tests demonstrated an improved convergence rate not matched on average by the classical Newton
method.

The goal of the present paper is to continue the work in~\cite{ScWi11} on simple algebraic systems, and to extend it to the context of general Banach spaces; in particular, nonlinear boundary value problems will be focused on, and an empirical investigation demonstrating the ability of the proposed approach to tame chaos in attractor boundaries will be provided in such situations. Furthermore, in contrast to the adaptive control mechanism in~\cite{ScWi11}, which is based on an intermediate step technique, we develop and test a pure prediction scheme in the present article. This will make it possible to compute the individual iterations much more efficiently. Indeed, this is most relevant in more complex applications such as in the numerical approximation of nonlinear ordinary and partial differential equations.

Finally, let us remark that there is a large application and
research area where methods related to the continuous version of
the Newton method are considered in the context of
nonlinear optimization. Some of these schemes count among the most
efficient ones available for special purposes; see, e.g., \cite{nocedal} and the references therein for details. 

\section{An Adaptive Newton Method} 

The aim of this section is to develop an adaptive Newton method based on a simple prediction strategy. To this end, we will first recall the continuous Newton ODE.

\subsection{Discrete vs.~Continuous Newton Method}

In order to improve the convergence behavior of the ({\em discrete}) Newton method~\eqref{eq:newton} in the case that the initial guess is far away from a root $x_{\infty}\in\Omega$, it is classical to consider a damped version of the Newton sequence. More precisely, given a possibly small~$t_n>0$, we consider the iteration
\begin{equation}
\label{damped}
x_{n+1}=x_n-t_n\F'(x_n)^{-1}\F(x_n).
\end{equation} 
Rearranging terms results in
\[
\frac{x_{n+1}-x_{n}}{t_n}=-(\F'(x_n))^{-1}\F(x_n), 
\]
we observe that~\eqref{damped} can be seen as the discretization of the initial value problem
\begin{equation}
\label{continuous}
\begin{cases} 
\begin{aligned}
\dot{x}(t)&=\NF_{\F}(x(t)), \qquad t\geq 0,\\
 x(0)&=x_0,
\end{aligned}
\end{cases}
\end{equation}
by the explicit Euler scheme with step size~$t_n$. Here, $\NF_{\F}(x)=-\F'(x)^{-1}\F(x)$ is the so-called Newton Raphson transform (NRT, for short; see~\cite{ScWi11}) of~$\F$. The system~\eqref{continuous} is called \emph{continuous Newton method}. It is noteworthy that, if $ \NF_{\F} $ is of class $C^1 $ on some neighborhood of $x_{\infty} \in \Omega$, then we have~$\mathsf{D}(\NF_{\F})(x_{\infty})=-\mathsf{Id}$. In particular, by the Poincar\'e-Ljapunow Theorem (see, e.g., \cite{12}) we conclude that each regular zero of $\F$ is located in an attracting neighborhood contained in $ \Omega $ when the NRT is applied. Furthermore, hoping that a sufficiently smooth solution of~\eqref{continuous} exists, and that $\lim_{t\to \infty}{x(t)}=x_{\infty}\in\Omega$ is well-defined with $ \F(x_{\infty})=0 $, one can readily infer that 
\begin{equation}
\label{cont}
\F(x(t))=\F(x_0)e^{-t}.
\end{equation}
The solvability of~\eqref{continuous} within the framework of Banach spaces has been addressed in~\cite{2,3}. Note that the trajectory of a solution of \eqref{continuous} either ends at the solution point $ x_{\infty} $ which is located closest to the initial value $x_{0} $, or at a some point close to a critical point $x_{c} $ with non-invertible derivative $\F' $, or at some point on the boundary $ \partial{\Omega} $ of the domain of $\F$; see \cite{10,11}. 

Given an approximation~$x_0\in\Omega$ of a solution~$x_\infty\in\Omega$ of~\eqref{eq:F0}, the basic idea in the design of the chaos-taming adaptive Newton scheme in this article is to provide some discrete dynamics which stay sufficiently close to the trajectories of the continuous Newton method leading to the root~$x_\infty$. Here, it is useful to take a global view: Instead of considering only one specific trajectory that transports an initial guess~$x_0$ to~$x_\infty$, we consider the global flow $ \Phi $ generated by the \emph{Newton-field} $ \NF_{\F} $. That is, for $ x\in \Omega $, we concentrate on the system
\begin{equation}
\label{global flow}
\begin{cases} 
\begin{aligned}
 \dot{\Phi}(t,x)&=\NF_{\F}(\Phi(t,x)), \qquad t\ge 0,\\
 \Phi(0,x)&= x.
\end{aligned}
\end{cases}
\end{equation}
For a given root~$x_\infty$ of~$\F$ we may now consider the set 
\begin{equation}
\label{domainofat}
\mathcal{A}(x_{\infty})=\left\{x_{0}\in \Omega:\, \lim_{t\to\infty}\|\Phi(t,x_0)-x_{\infty}\|_X=0\right\}
\end{equation}
of all points which belong to trajectories of~\eqref{global flow} leading to~$x_\infty$. We note that the discrete dynamics as described by the Newton iteration~\eqref{damped} are based on possibly small but non-infinitesimal step sizes~$t_n>0$. In particular, the discrete iterates approximate the continuous trajectories from~\eqref{global flow} and may therefore jump back and forth between them. The chaotic behavior of the discrete Newton method is tamed as long as the iterates stay within the {\em same} attractor~$\mathcal{A}(x_\infty)$. Here, it is important to note that this is achievable in principle as long as the step sizes~$t_n>0$ are sufficiently small. Indeed, provided that~$\Phi$ is continuous and that~$x_\infty$ is a regular zero of~$\F$ which is contained in an attractive neighbourhood~$B_\eta(x_\infty)\subset\mathcal{A}(x_\infty)$, for some~$\eta>0$, this simply follows from the fact that~$\mathcal{A}(x_\infty)$ is an open set: To see this, we choose any $ x_0 \in  \mathcal{A}(x_{\infty})$; then, there exists $t>0$ such that $ \Phi(t,x_0)\in B_{\eta}(x_{\infty}) $. The openness of $B_{\eta}(x_{\infty})$ together with the continuity of $\Phi $ implies the existence of some $\varepsilon, \delta>0 $ such that
$\Phi(t,B_{\delta}(x_0))\subset B_{\varepsilon}(\Phi(t,x_0))\subset B_{\eta}(x_\infty)$, i.e., 
$B_{\delta}(x_{0})\subset \mathcal{A}(x_{\infty}) $.


\subsection{A Prediction Strategy}

In this section we discuss the linearization of the Newton-field $ \NF_{\F} $ which will serve as a prediction strategy of the exact trajectories given in \eqref{global flow}. We propose an adaptive path-following algorithm in such a way that, for a given initial guess $ x_{0} \in \mathcal{A}(x_{\infty}) $, the iterates $\{x_{n}\}_{n\in\N}$ presumably stay within $ \mathcal{A}(x_{\infty}) $. 

To simplify matters we fix $ x(t)=\Phi(t,x_{0}) $ for $x_{0} \in \mathcal{A}(x_{\infty})$ and denote by $\hat{x} $ the linearization at $t=0$, $x(0)=x_0$, i.e.
\begin{equation}
\label{linearization}
\hat{x}(t)=x_{0}+t\dot{x}(0).
\end{equation}
By the openess of~$\mathcal{A}(x_\infty)$ we note that, for sufficiently small $t>0 $, there holds $\hat{x}(t) \in \mathcal{A}(x_{\infty}) $. 
Let us focus on the distance between $x(t)$ and its linearization~$\hat{x}(t)$:
In view of \eqref{continuous} and \eqref{linearization} we have that
\begin{equation}
\label{ab}
\begin{aligned}
\hat{x}(t)-x(t)&=\int_0^t\left(\dot{\hat{x}}(s)-\dot{x}(s)\right)\ds\\
&=\int_{0}^{t}{(\NF_{\F}(x_0)-\NF_{\F}(x(s)))\ds}\\
&=\int_{0}^{t}{(\NF_{\F}(x_0)-\NF_{\F}(x_0)e^{-s})\ds}+\int_{0}^{t}{(\NF_{\F}(x_0)e^{-s}-\NF_{\F}(x(s)))\ds}\\
&=\NF_{\F}(x_0)(t+e^{-t}-1)+I(t),
\end{aligned}
\end{equation}
with
\[
I(t)=\int_{0}^{t}{(\NF_{\F}(x_0)e^{-s}-\NF_{\F}(x(s)))\ds}.
\]
Using~\eqref{cont} we obtain
\[
\F'(x_0)^{-1}\frac{\d}{\ds}\F(x(s))
=\F'(x_0)^{-1}\frac{\d}{\ds}\left(\F(x_0)e^{-s}\right)
=-\F'(x_0)^{-1}\F(x_0)e^{-s}
=\NF_{\F}(x_0)e^{-s}.
\]
Thus, recalling~\eqref{continuous}, we get 
\begin{align*}
I(t)&=\int_{0}^{t}{\left(\F'(x_0)^{-1}\frac{\d}{\ds}\F(x(s))-\dot{x}(s)\right)}\ds\\
&=\F'(x_0)^{-1}(\F(x(t))-\F(x_0))-x(t)+x_0.
\end{align*}
A Taylor expansion for $ \F $ about~$x_0$ is given by
\begin{align*}
\F'(x_0)^{-1}&(\F(x(t))-\F(x_0))\\
&=\F'(x_0)^{-1}\left(\F(x_0)+\F'(x_0)(x(t)-x_0)+\mathcal{O}(\|x(t)-x_0\|_X^2)-\F(x_0)\right)\\
&=x(t)-x_0+\mathcal{O}(\|x(t)-x_0\|_X^2).
\end{align*}
In particular, we see that $ I(t)=\mathcal{O}(\|x(t)-x_0\|_X^2) $. Going back to \eqref{ab} we arrive at
\[
\hat{x}(t)-x(t)=\NF_{\F}(x_0)(t+e^{-t}-1)+\mathcal{O}(\|x(t)-x_0\|_X^2).
\]
We see that by neglecting the term $ \mathcal{O}(\|x(t)-x_0\|_X^2) $, the expression $ \NF_{\F}(x_0)(t+e^{-t}-1) $ is a computable quantity  and can be used as an error indicator in each iteration step. Moreover, using that~$e^{-t}=1-t+\frac12t^2+\mathcal{O}(t^3)$, it follows that
\[
\hat{x}(t)-x(t)=\frac12t^2\NF_{\F}(x_0)+\mathcal{O}(t^3)+\mathcal{O}(\|x(t)-x_0\|_X^2).
\]
Thence, fixing a tolerance $ \tau>0 $ such that
\[
\begin{aligned}
\tau &= \norm{\hat{x}(t)-x(t)}_X
=\frac{t^2}{2}\norm{\NF_{\F}(x_0)}_{X}+\mathcal{O}(t^3)+\mathcal{O}(\|x(t)-x_0\|_X^2),
\end{aligned}
\]
and ignoring the higher order approximation terms, motivates the following adaptive step size control procedure for the Newton iteration:\\

\begin{algorithm}~\label{algorithm}
Fix a tolerance $ \tau>0 $.
\begin{enumerate}[i)]
\item  Start the Newton iteration with an initial guess $ x_{0} \in \mathcal{A}(x_{\infty}) $.
\item In each iteration step $ n=0,1,2,\ldots $, compute
\begin{equation}
\label{neun}
t_{n}=\min\left(\sqrt{\frac{2\tau}{\norm{\NF_\F(x_{n})}_{X}}},1\right).
\end{equation}
\item Compute~$x_{n+1}$ based on the Newton iteration~\eqref{damped}, and go to the next step $ n\leftarrow n+1 $. 
\end{enumerate}
\end{algorithm}

\begin{remark}
The minimum in~\eqref{neun} is chosen such that~$t_n=1$ whenever possible, in particular, close to a root. This will retain the celebrated quadratic convergence property of the Newton scheme (provided that the corresponding root is simple). 
\end{remark}

\begin{remark}
Since we fix $ \tau $ a priori it might happen that the step size $t_n$ from~\eqref{neun} may be too large in the sense that the Newton sequence $\{x_{n}\}_{n\in\N } $ leaves the attractor $ \mathcal{A}(x_{\infty})$. Indeed, our Algorithm~\ref{algorithm} obviously lacks a correction strategy for the predicted step size. This is in contrast to the references~\cite{5,ScWi11} in the context of finite-dimensional algebraic systems, where the reduction of the step size may possibly be corrected in order for the iterates to stay within $ \mathcal{A}(x_{\infty})$. Evidently, however, a possible repeated reduction of the step size may strongly increase the computational complexity. Indeed, in view of solving nonlinear operator equations in infinite dimensional Banach spaces (by means of suitable discretization schemes), which are of interest in this work, a corresponding procedure might become unfeasibly expensive in practical applications. 
\end{remark}

\subsection{A Convergence Result}
We close this section by casting the rather geometrically inspired prediction path-following Algorithm~\ref{algorithm} into a framework of a global analysis. There are various approaches that have been presented in the literature. Here, we follow along the lines of~\cite{5}, and show that the residuum $\F(x_n)\to 0$ as~$n\to\infty$ if certain (quite strong) conditions hold. Specifically, we assume that, for given~$\tau>0$ and~$x_0\in \Omega$, the Newton sequence~$\{x_n\}_{n\ge 0}$ defined in~\eqref{damped} with~$t_n$ from~\eqref{neun} satisfies the following properties:
\begin{enumerate}[(A)]
\item The sequence~$\{x_n\}_{n\ge 0}$ is well-defined, i.e., in particular, for any~$n\ge 0$, we have that~$x_n\in\Omega$, and~$\F'(x_n)$ is invertible.
\item There exists a constant~$\hat{K}>0$ such that $\norm{\F'(x_n)^{-1}}_{Y\to X}\leq \hat{K} $ for all~$n\ge 0$.
\item There is a compact set~$M\subset\Omega$ as well as a constant~$\tilde{K}>0$ such that the piecewise linear trajectory connecting the points~$x_0, x_1, x_2,\ldots$ is contained in~$M$ and such that
$ \norm{\F'(x)-\F'(y)}_{X\to Y}\leq \tilde{K}\norm{x-y}_{X}$ for all  $x,y\in M $.
\end{enumerate}


\begin{theorem}\label{pr}
Let~$x_0\in\Omega$, and suppose that there exists~$\tau_0>0$ such that the properties~{\rm (A)--(C)} above are fulfilled for any~$\tau\le\tau_0$.
Then, for
\begin{equation}
\label{tau'}
0<\tau<\min\left\{\tau_0,\frac{2}{\hat{K}^2\tilde{K}^2}\inf_{n\ge 0}\|\NF_{\F}(x_n)\|_X^{-1},\hat{K}^{-1}\tilde{K}^{-1}\right\},
\end{equation}
the adaptive Newton iteration~\eqref{damped}, with~$t_n$ from~\eqref{neun}, $n\ge 0$, converges, i.e., it holds that~$\lim_{n\to\infty}\norm{\F(x_n)}_X=0$. 
\end{theorem}

\begin{remark}
We note that, for all~$n\ge 0$, we have that
\begin{equation}\label{eq:2014_1}
\|\NF_{\F}(x_n)\|_X\le \norm{\F'(x_n)^{-1}}_{Y\to X}\norm{\F(x_n)}_Y\le \hat{K}\sup_{x\in M}\norm{\F(x)}_Y<\infty.
\end{equation}
The last inequality follows from the fact that~$M$ is compact and that the mapping~$x\mapsto\|\F(x)\|_Y$ is continuous on~$M$. As a consequence, the set~$\{\|\F(x)\|_Y:\, x\in M\}$ is compact in~$\mathbb{R}$, and hence bounded and closed. In particular, the supremum in~\eqref{eq:2014_1} is attained and bounded. Thus, if~$\F\not\equiv 0$ on~$M$,
\[
\inf_{n\ge 0}\|\NF_{\F}(x_n)\|_X^{-1}\ge \hat{K}^{-1}\left(\sup_{x\in M}\|\F(x)\|_Y\right)^{-1}>0.
\]
Especially, it is possible to choose~$\tau>0$ in~\eqref{tau'}.
\end{remark}

\begin{remark}
We note that the assumptions in Theorem~\ref{pr} are of a theoretical nature and difficult to check in general. From a heuristic point of view, however, our result illustrates that convergence of the Newton sequence to a zero of~$\F$ is reasonable to achieve, provided that~$\tau$ is chosen sufficiently small.
\end{remark}

\begin{proof}[Proof of Theorem~\ref{pr}]
Let~$\tau>0$ satisfy~\eqref{tau'}. Then, we choose~$\epsilon>0$ such that
\[
0<\tau(1+\epsilon)^2\le\min\left\{\tau_0,\frac{2}{\hat{K}^2\tilde{K}^2}\inf_{n\ge 0}\|\NF_{\F}(x_n)\|_X^{-1},\hat{K}^{-1}\tilde{K}^{-1}\right\}.
\]
By the mean value theorem we have
\[
\F(x_1)-\F(x_0)=\F(x_0+t_0\NF_{\F}(x_0))-\F(x_0)=\left(\int_{0}^{t_0}{\F'(x_0+s\NF_{\F}(x_0))\ds}\right)\NF_{\F}(x_0).
\]
Hence,
\[
\F(x_1)=\F(x_0)(1-t_0)+\left(\int_{0}^{t_0}{\left(\F'(x_0+s\NF_{\F}(x_0))-\F'(x_0)\right)\ds}\right)\NF_{\F}(x_0).
\]
In particular, recalling condition~(C) above, we notice that the previous integrals are all well-defined. 
By definition, we have that~$t_0\in(0,1]$, and thus, employing the triangle inequality, we obtain the estimate
\begin{equation}\label{eq:13}
\begin{aligned}
\norm{\F(x_1)}_{Y}
&\le(1-t_0)\norm{\F(x_0)}_Y\\
&\quad+\norm{\int_{0}^{t_0}{\left(\F'(x_0+s\NF_{\F}(x_0))-\F'(x_0)\right)\ds}}_{X\to Y}\norm{\NF_{\F}(x_0)}_X\\
& \leq (1-t_0)\norm{\F(x_0)}_{Y}+\frac{t_0^2}{2}\tilde{K}\norm{\NF_{\F}(x_0)}_{X}^2
\leq \gamma_0\norm{\F(x_0)}_Y,
\end{aligned}
\end{equation}
where
\[
\gamma_0=1-t_0+\frac{t_0^2}{2}K\norm{\NF_{\F}(x_0)}_{X}, 
\]
for  $ K=\hat{K}\tilde{K}$. In order to estimate~$\gamma_0$, we consider two cases:

\begin{enumerate}[{Case}~1:]
\item Let first
\[
\frac{2\tau}{\norm{\NF_\F(x_0)}_X}\ge 1.
\]
Then, $t_0=1$ in~\eqref{neun}, and~$\norm{\NF_\F(x_0)}_X\le 2\tau$. Therefore,
\[
\gamma_0=\frac12K\norm{\NF_\F(x_0)}_X\le K\tau.
\]
Using that
\begin{equation}\label{eq:2014}
\tau(1+\epsilon)<\tau(1+\epsilon)^2\le K^{-1},
\end{equation}
results in 
\[
\gamma_0<\frac{1}{1+\epsilon}<1.
\]
\item If secondly,
\[
\frac{2\tau}{\norm{\NF_\F(x_0)}_X}< 1,
\]
then
\[
t_0=\sqrt{\frac{2\tau}{\norm{\NF_\F(x_0)}_X}}
\ge\sqrt{2\tau\inf_{n\ge 0}\|\NF_{\F}(x_n)\|_X^{-1}}.
\]
Noticing that
\[
\tau\le \frac{2}{K^2(1+\epsilon)^2}\inf_{n\ge 0}\|\NF_{\F}(x_n)\|_X^{-1},
\]
or equivalently,
\[
\inf_{n\ge 0}\|\NF_{\F}(x_n)\|_X^{-1}\ge\frac{\tau K^2(1+\epsilon)^2}{2},
\]
we arrive that 
\[
t_0\ge\tau K(1+\epsilon).
\]
In this way, we obtain
\[
\gamma_0=1-t_0+K\tau\le 1-\tau K(1+\epsilon)+K\tau\le 1-K\tau\epsilon.
\]
Recalling~\eqref{eq:2014}, we see that~$0<K\tau\epsilon<\epsilon(1+\epsilon)^{-1}<1$.
\end{enumerate}
In summary, we see that~$\gamma_0\le q$, where~$q=\max\left(1-K\tau\epsilon,(1+\epsilon)^{-1}\right)\in(0,1)$. It follows from~\eqref{eq:13} that~$\norm{\F(x_1)}_X\le q\norm{\F(x_0)}_X$. By induction, we conclude that
\[
\norm{\F(x_n)}_X\le q^n \norm{\F(x_0)}_X\to 0,
\]
with~$n\to\infty$. This completes the proof.
%
\end{proof}

\section{Applications}

The purpose of this section is to illustrate Algorithm~\ref{algorithm} by means of a number of examples. In particular, we will focus on nonlinear algebraic systems and on differential equations.

\subsection{Algebraic equations}
Let us look at two algebraic problems. The first one is a cubic polynomial equation on~$\mathbb{C}$ (identified with~$\mathbb{R}^2$) with three separate zeros, and the second example is a challenging benchmark problem in~$\mathbb{R}^2$.

\begin{example}\label{ex:alg1}
We consider the function
\begin{equation}
\label{Example1}
\F:\C\rightarrow \C, \qquad z\mapsto \F(z)=z^3-2z-4,
\end{equation}
with the three zeros 
\[
Z_{\F}=\{(2,0),(-1,1),(-1,-1)\}\subset\mathbb{C}. 
\]
We observe that~$\F'$ vanishes at~$\left(\pm\sqrt{\nicefrac{2}{3}},0\right)$. This causes large updates in the Newton iteration close to those points, and hence, a source of potential chaos has been generated by applying the NRT; cf.~\cite[Example~2]{ScWi11}. In order to discuss the behavior of the Newton method for this example, let us first focus on the vector fields corresponding to $ \F  $ and $ \NF_\F $; see Figure~\ref{bild3} left and right, respectively. One can clearly see that the root $ (2,0)\in Z_{\F} $ is repulsive for~$\F$. Moreover, the zeros $\{(-1,1),(-1,-1)\} \in Z_{\F} $ of $\F$ show a curl. For~$\NF_\F$ the situation is completely different: All the three roots are attractive, and the vectors point directly to the three roots of $\F$.  Therefore, the NRT $\NF_{\F} $ can be used to transport an initial guess $x_0 \in \mathcal{A}(x_{\infty})$ arbitrarily close to a root $x_{\infty}$. In the given example, we observe that the vector direction field is divided into three different sectors for~$\NF_\F$, which are the attractors for the initial value problem \eqref{continuous}.

\begin{figure}
\includegraphics[width=0.45\textwidth]{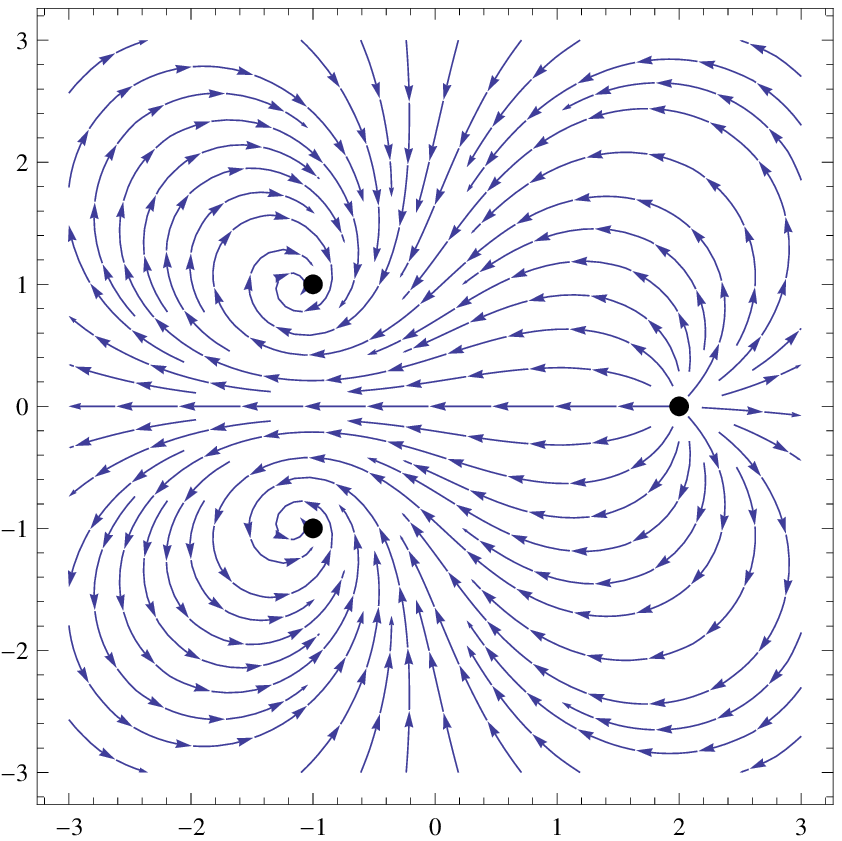}
\hfill
\includegraphics[width=0.45\textwidth]{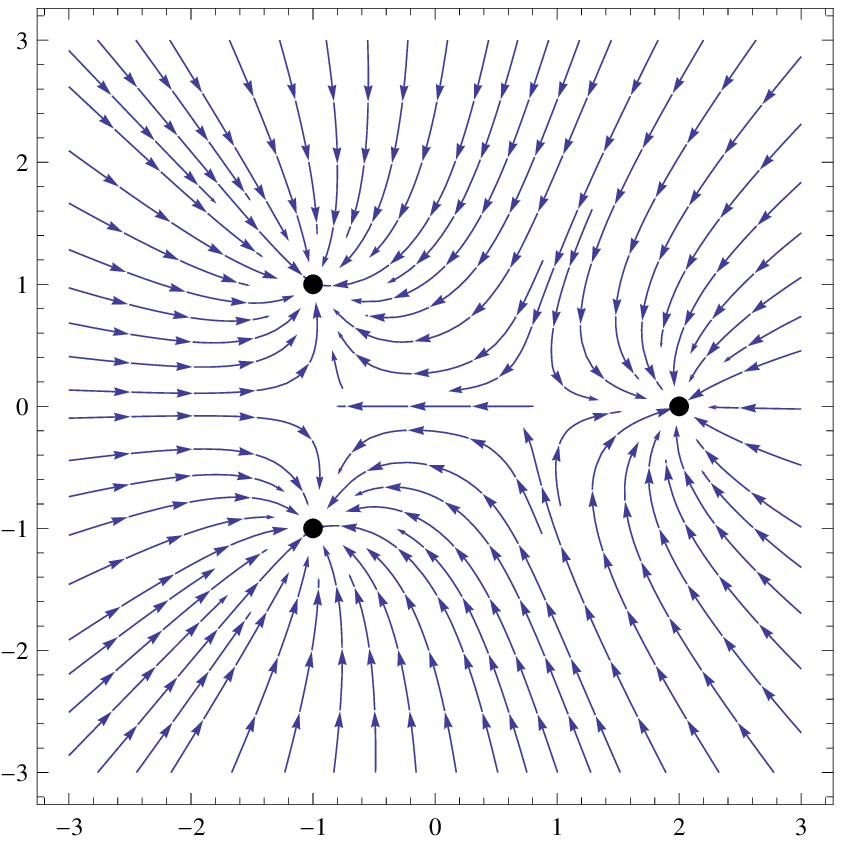}
\caption{The direction fields corresponding to $ \F(z)=z^3-2z-4 $ (left) and to the NRT (right).}
\label{bild3}
\end{figure}

In Figure~\ref{performance1} we display the behavior of the classical (with step size~$t_n=1$), the continuous, and the adaptive Newton method (with $\tau=0.05$ and~$t_n$ from~\eqref{neun}), for the initial point $ x_{0}=(0.08,0.55) $. We see that, while the classical solution shows large updates and thereby leaves the original attractor, the iterates corresponding to the adaptive Newton method follow the exact solution (which is approximated by a numerical reference solution with $ t\ll 1 $) quite closely and approach the "correct" zero. 

\begin{figure}[htp]
\begin{center}
\includegraphics[width=0.75\textwidth]{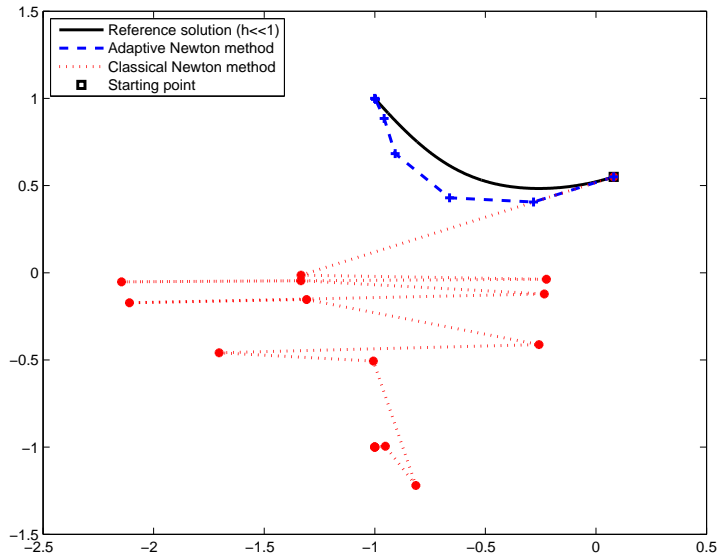}
\end{center}
\caption{Performance of the classical Newton and the Newton method with adaptive step size control (with $\tau = 0.05 $) for the starting point $ x_0=(0.08,0.55)$.}
\label{performance1}
\end{figure}

In order to visualize the domains of attraction of different Newton schemes, we compute the zeros of $\F$ by sampling initial values on a $1001\times 1001 $ grid in the domain $ [-5,5]\times [-5,5] $. In Figure~\ref{bild10}, we show the fractal generated by the traditional Newton method with constant step size~1 (left) as well as the corresponding plot for the damped Newton scheme with constant step size~$0.72$. We observe that the damped Newton method is able to control the chaos to some extent, however, there are still relatively large fractal areas. Furthermore, in Figure~\ref{bild8}, we use adaptive step size control based on Algorithm~\ref{algorithm} by setting $ \tau=0.1 $ (left) and~$\tau=0.001$ (right). The chaotic behavior caused by the singularities of~$\F'$ is clearly tamed by the adaptive Newton method. 

\begin{figure}
\includegraphics[width=0.45\textwidth]{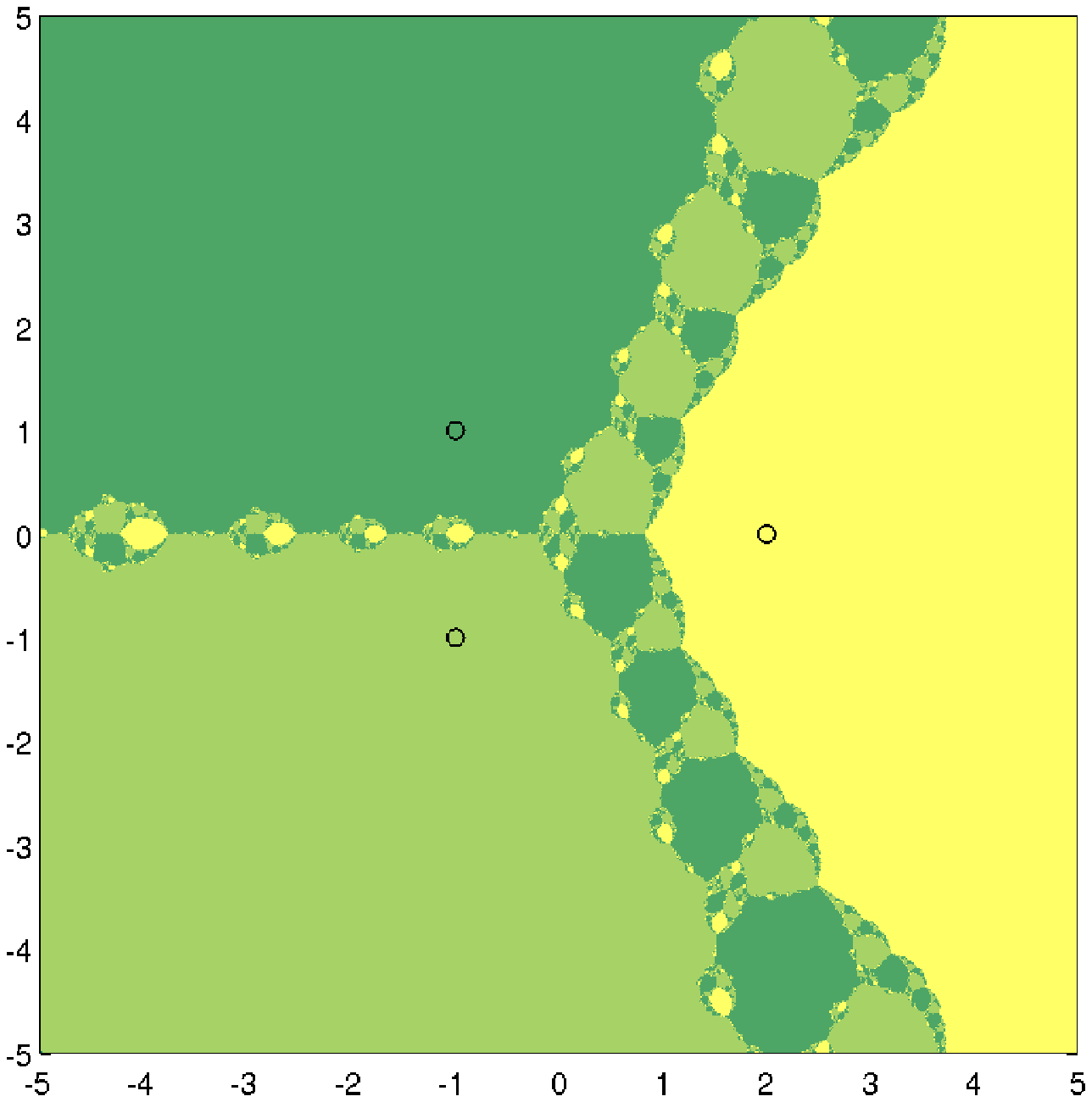}
\hfill
\includegraphics[width=0.45\textwidth]{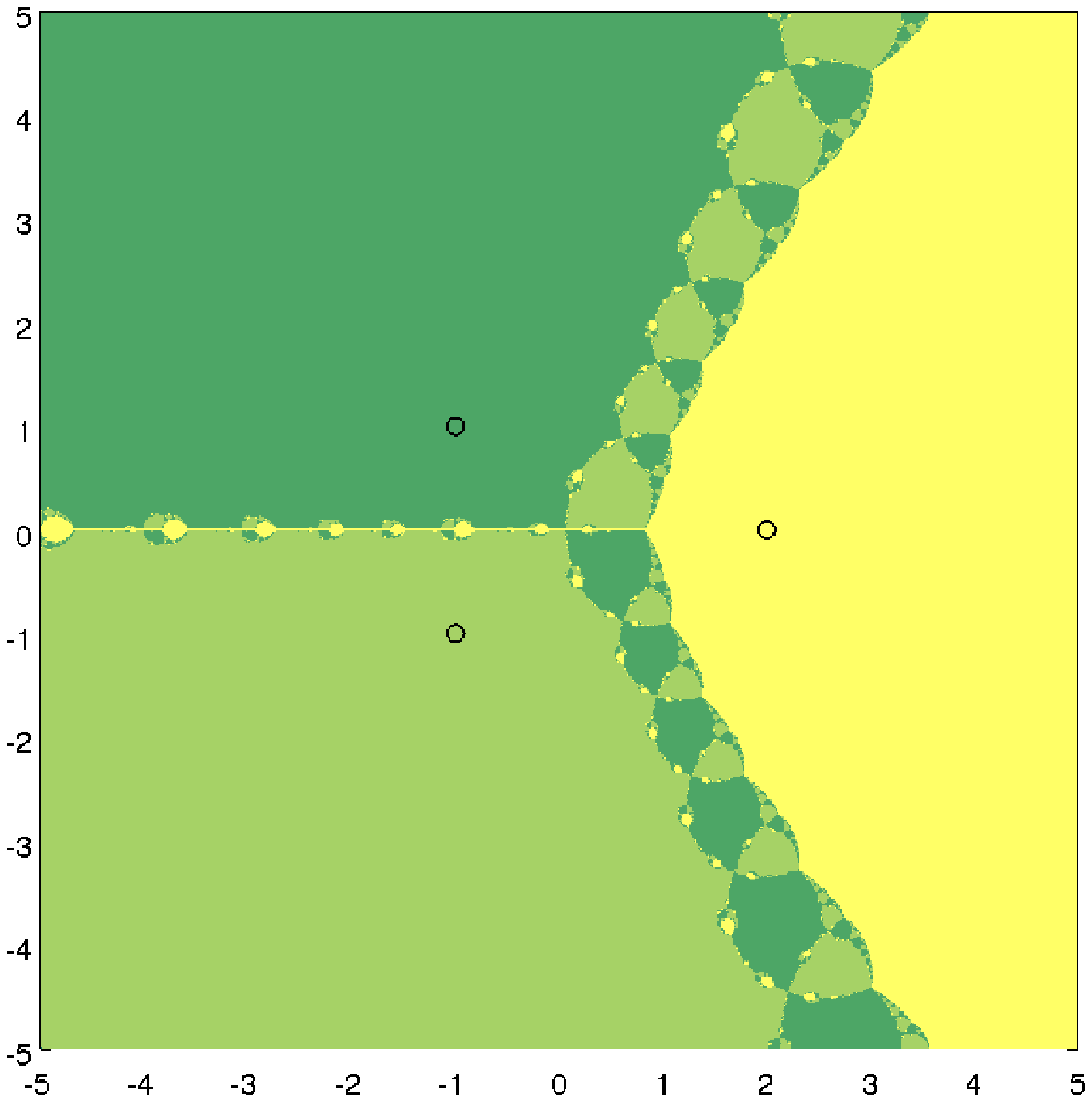}
\caption{The basins of attraction for $ z^3-2z-4=0 $ by the Newton method: The classical scheme on the left (i.e., $t=1$), and on the right with a fixed reduced step size ($t=0.72$). Three different colors distinguish the three basins of attraction associated with the three solutions (each of them is marked by a small circle).}
\label{bild10}
\end{figure}

\begin{figure}[htp]
\includegraphics[width=0.45\textwidth]{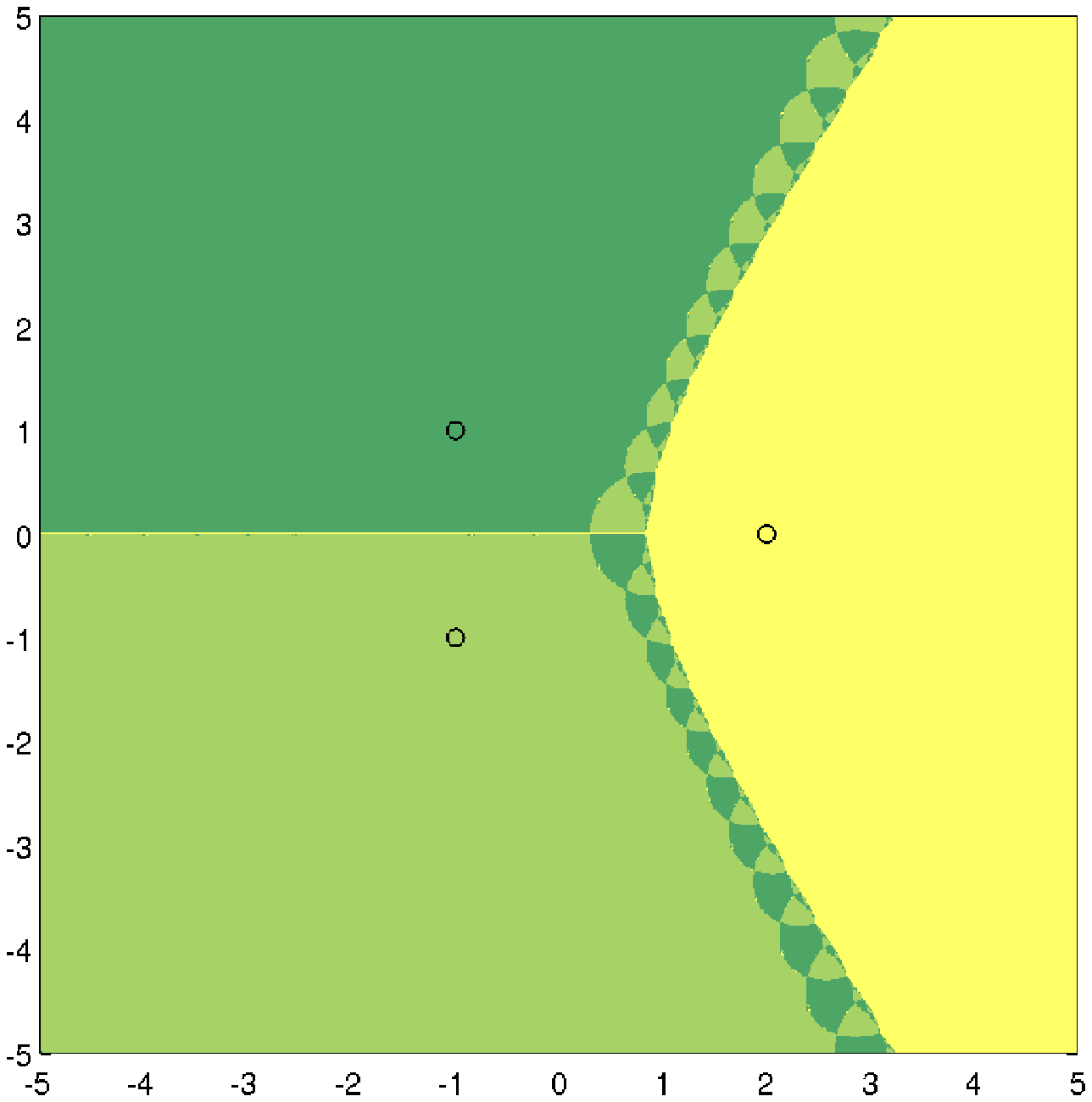}
\hfill
\includegraphics[width=0.45\textwidth]{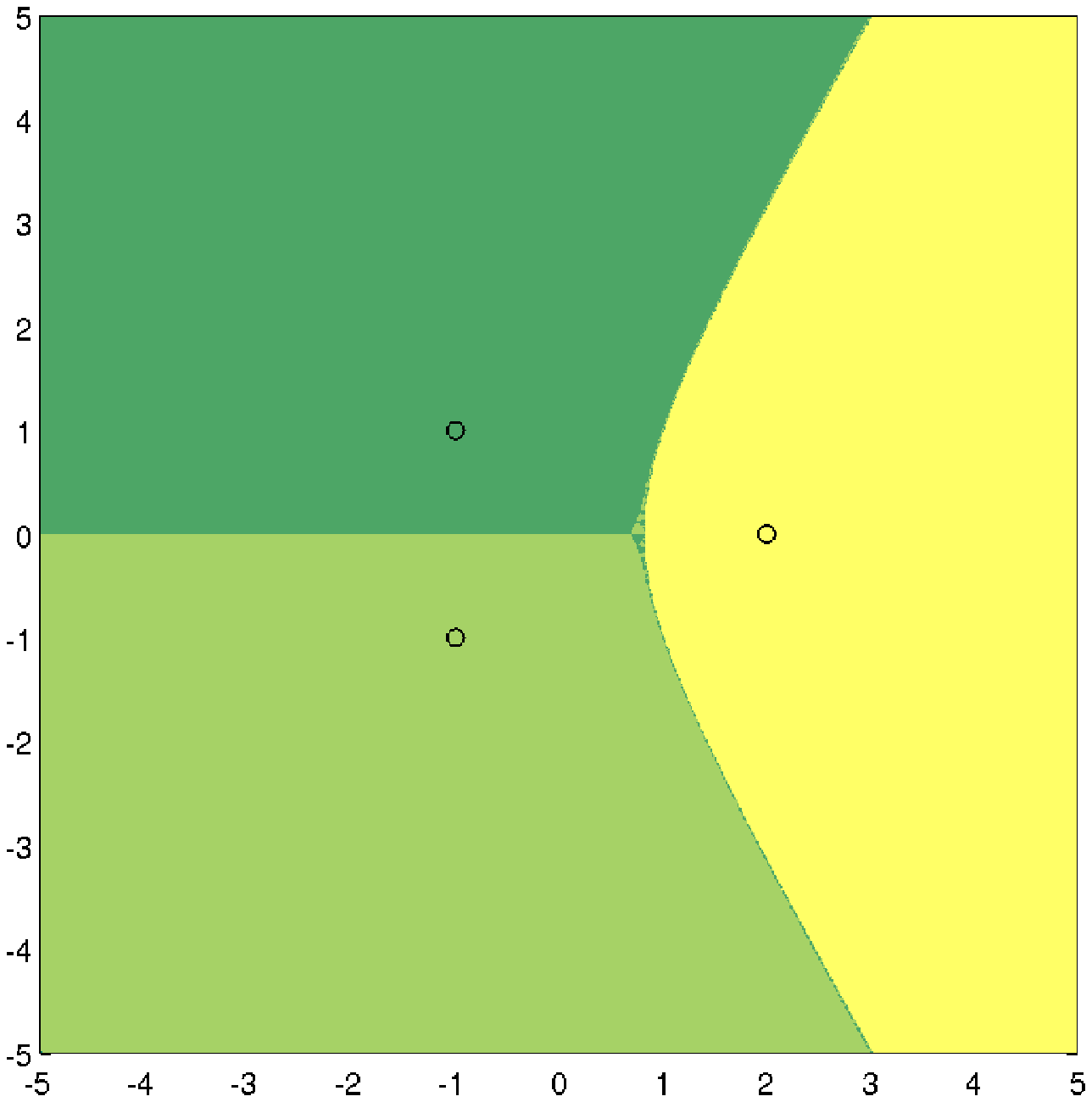}
\caption{Attractors for $ z^3-2z-4=0 $ by the Newton method. On the left with step size control for ($\tau=0.1$) and on the right for $(\tau=0.001)$.}
\label{bild8}
\end{figure}

Comparing the statistics resulting from a step size control computation with $ \tau =0.1 $ with the corresponding results for a fixed step size underlines the superiority of the proposed approach; see the performance data in Table~\ref{performancetable1}. 
The information is based on $ 10^4 $ starting values in the domain $ [-5,5]\times [-5,5] $. We list the percentage of convergent iterations, the average number of iterations necessary to obtain an absolute accuracy of at least $10^{-8} $, and the average convergence rate defined as follows: The error in the $n$-th iteration, that is
\[
e_{n}=\norm{x_{\infty}-x_n}, \qquad x_{\infty} \in Z_{\F},
\]  
is supposed to satisfy a relation of the form 
\begin{equation}
\label{least}
e_n=ce^{\rho}_{n-1}\qquad \Leftrightarrow \qquad \ln(e_n)=C+\rho \ln(e_{n-1}),\quad n\in  \N,
\end{equation}
for a constant~$\rho$. This is the rate of convergence, which, for $ n \to \infty $, will typically tend to a stable limit. Clearly, due to finite resources, we can determine $ \rho $ only empirically, i.e., we denote by $ \tilde{\rho} $ the convergence rate that we will obtain by applying a least squares approximation to \eqref{least} (averaged over all computed iterations) for the unknown parameters $ \rho $ resp.~$ C $. A starting value $ x_0 $ is called convergent if it is in fact convergent and, additionally, approaches the "correct zero", i.e. the zero that is located in the same exact attractor as the initial value $ x_0 $. To decide whether or not the starting value $x_0$ approaches the correct zero, we simultaneously compute a reference solution $x_{\text{ref}}$ using a fixed step size $ t\ll1 $. Our results demonstrate, in contrast to the Newton method with fixed step size, that the rate of convergence in the adaptive approach is nearly quadratic, and that the number of convergent iterations is close to $100$\%.

\begin{table}[htp]
\caption{Performance data for Example~\ref{ex:alg1} on $ [-5,5]\times [-5,5] $.}
\begin{center}
\begin{tabular}{@ {}*{4}{ l}@ {}}\toprule
 & Step size $h\equiv 1$ & Step size $ \equiv 0.72 $ & Adapt. $ \tau=0.1 $\\
 \midrule
 Average nr. of iterations & $21.4$ & $27$ & $14$  \\
 Average step size & $1$ & $0.72$ & $0.72$  \\
 \% of convergent iterations  & $87.7\%$ & $92\%$ & $96.5\%$ \\
 Average rate $\tilde{\rho}$ & $1.72$ & $0.945$ & $1.89$ \\
 \bottomrule
\end{tabular}
\end{center}
\label{performancetable1}
\end{table}
\end{example}


\begin{figure}[htp]
\includegraphics[width=0.45\textwidth]{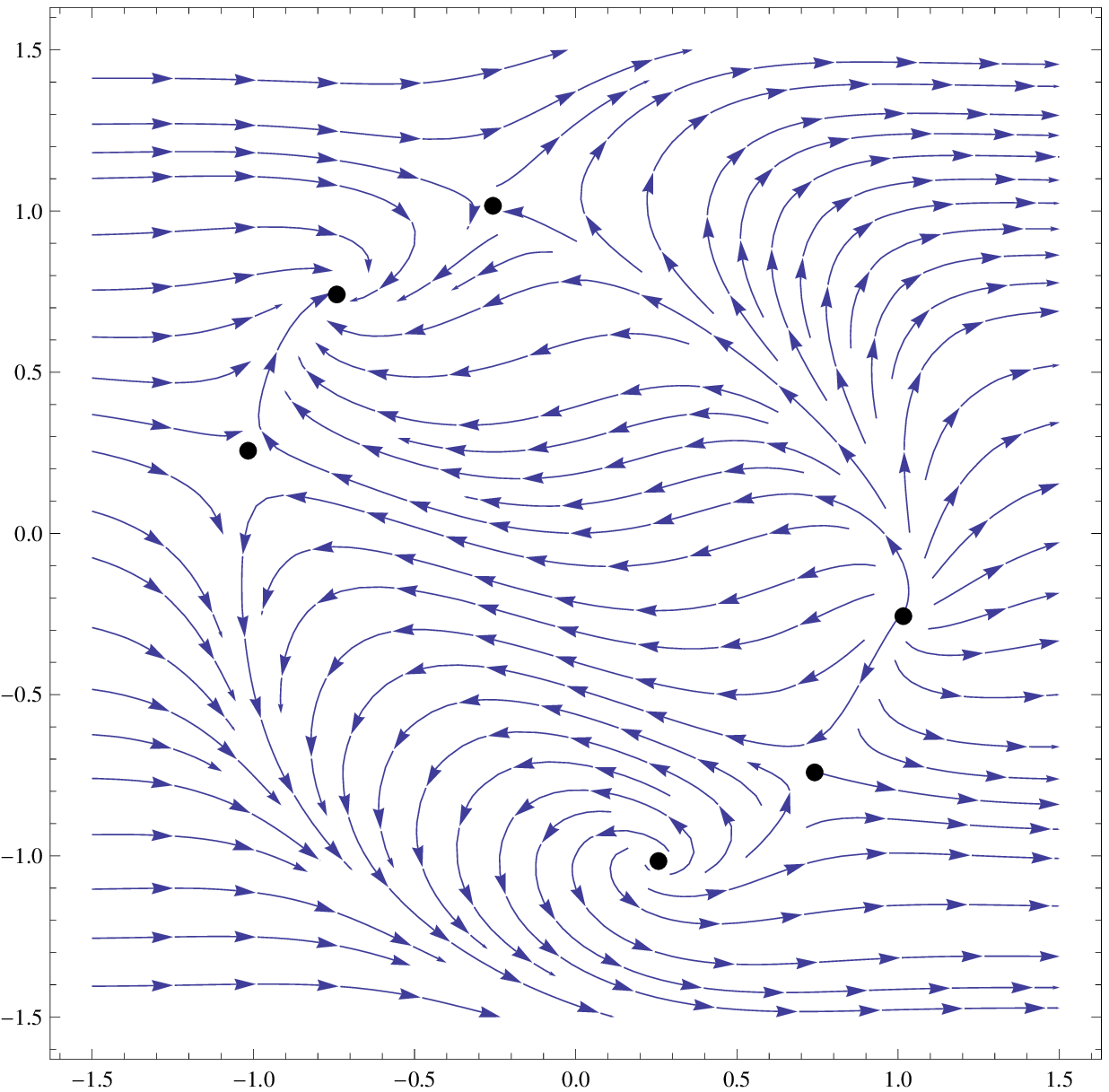}
\hfill
\includegraphics[width=0.45\textwidth]{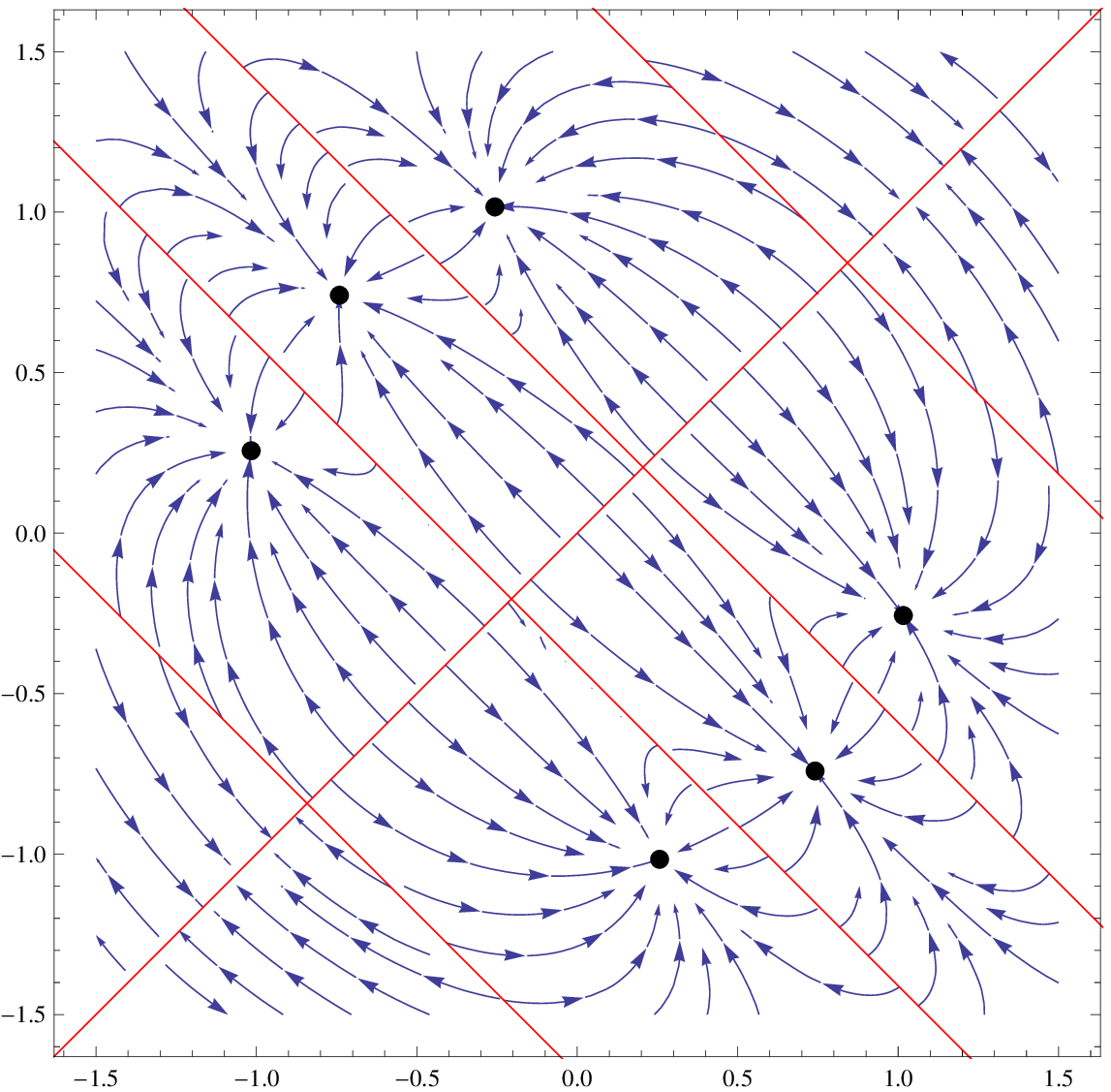}
\caption{The direction field for $\F$ resp.~of the NRT over the domain $ \Omega = [-1.5,1.5]^2$.}
\label{flowdeufelhard}
\end{figure}

\begin{example}\label{ex:alg2} 
The second example is a benchmark $2\times2$ algebraic system from~\cite{5}. Consider the function
\begin{equation}
\label{expsin}
\F:\Omega\subset \R^2 \rightarrow \R^{2}, \qquad \F(x,y)=\begin{pmatrix}\exp(x^2+y^2)-3\\x+y-\sin(3(x+y))\end{pmatrix},
\end{equation}
with $ \Omega=[-1.5,1.5]^{2} $.
First of all we notice that the set where the Jacobian of~$\F$ becomes singular is given by the straight lines
\begin{equation}
\label{singular}
\{y=x\},   \quad \text{and} \quad \left\{y=-x\pm \frac{1}{3}\arccos\left(\frac{1}{3}\right)\pm\frac{2}{3}\pi k, \ k\in \N_{\geq0}\right\}.
\end{equation}
The set $\Omega$ contains exactly six different roots of $\F$, which all become  locally attractive when applying the NRT; see Figure~\ref{flowdeufelhard} (right). However, for these six roots, we have six different basins of attraction, which are separated by the straight lines given in \eqref{singular}. In Figure~\ref{flowdeufelhard} (right) the red lines indicate the critical interfaces where the Jacobian becomes singular.  

Before we apply the Newton method to this example let us point out an important fact: The continuous Newton ODE is obviously not able to lead an initial guess $x_0$ to a root of $\F$ when we start in a separated subdomain where no root is located. The present example nicely underlines this effect when we focus on the top right or the bottom left part of the domain $\Omega $ (see Figure~\ref{flowdeufelhard} (right)). In particular, when starting with an initial guess located in a domain where we have no root for $\F$, the corresponding Newton path ends at a critical point. This is potentially different when we apply the discretized version. In fact, starting in a subdomain without a root does not necessarily imply that the Newton method will be unable to find a root of $\F$ since the {\em discrete} sequence may indeed cross critical interfaces. If we choose $\tau\ll1 $, however, the Newton sequence is close to its corresponding continuous Newton path. This indicates that retaining a certain amount of chaos (i.e., choosing~$\tau$ not too small) in the discrete Newton iteration might even increase the domain of convergence. This is particularly important when no a priori information on the location of the zeros is available. In Figure~\ref{fig:expsinbasins} we display the domains of attraction.  Note that the dark blue shaded part indicates the domain where the iterations fail to converge. We clearly see that step size control is able, on the one hand, to tame the chaotic behavior of the iteration and, on the other hand, to enlarge the domain of convergence. Table~\ref{performancetableexpsin} presents the performance data for the classical and the adaptive Newton method by sampling $10^4$ initial values  on the domain $\Omega=[0,1.5]\times [-1.5,0] $. Again, the favorable convergence features of the adaptive approach become evident.

\begin{figure}[htp]
\includegraphics[width=0.45\textwidth]{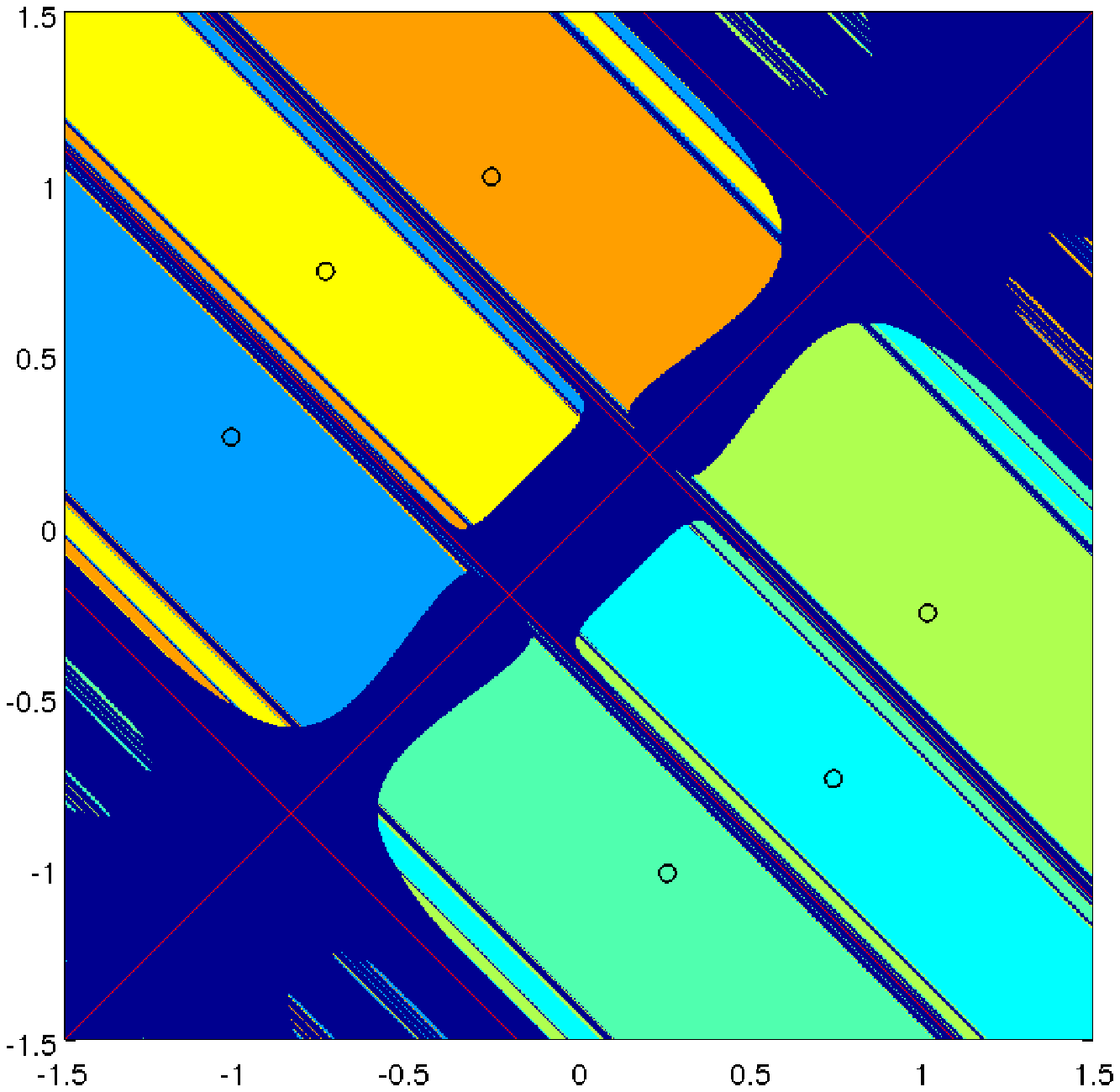}
\hfill
\includegraphics[width=0.45\textwidth]{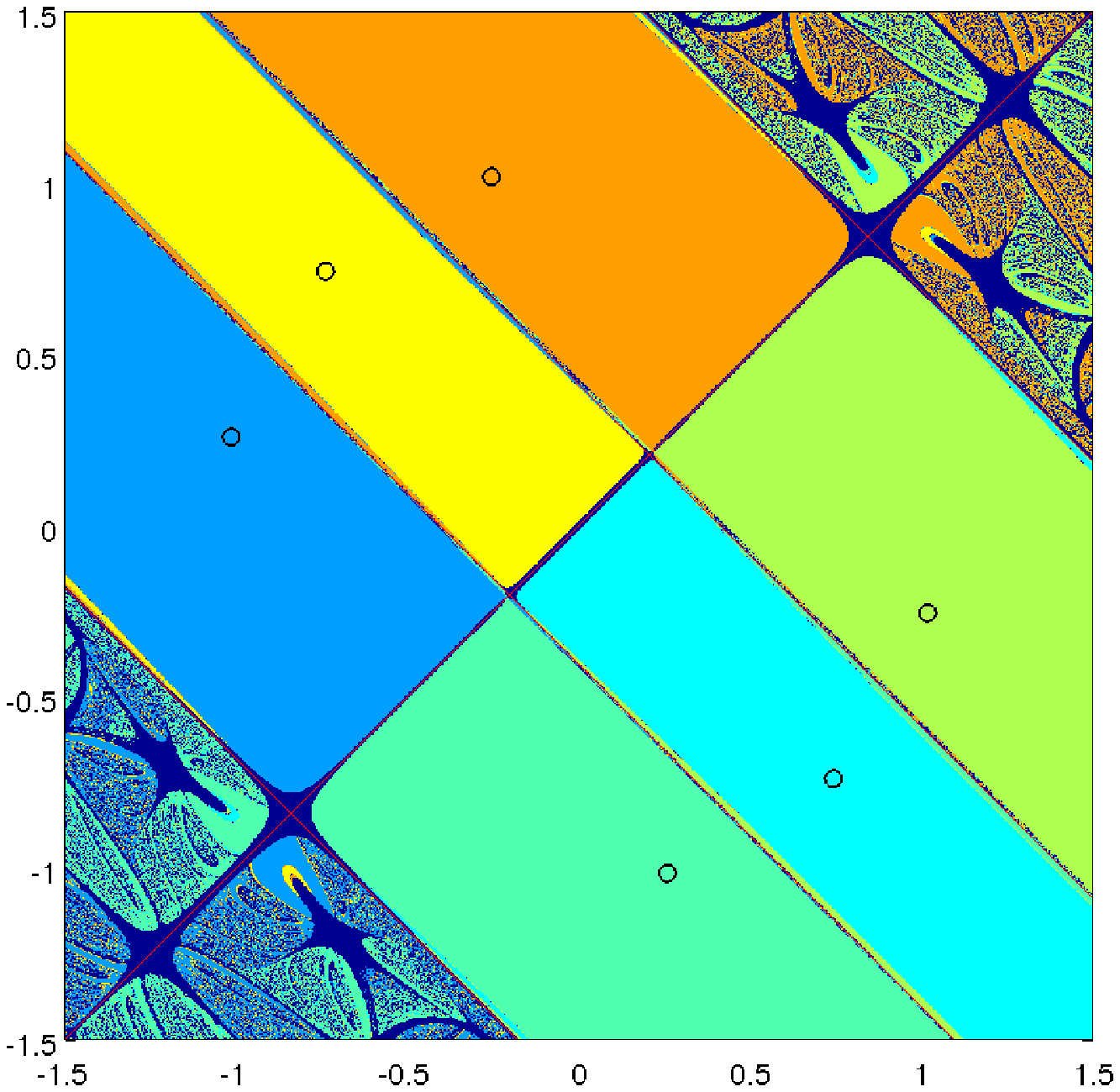}
\caption{Classical Newton method (left) and adaptive Newton method with $\tau = 0.1 $ (right).}
\label{fig:expsinbasins}
\end{figure}

\begin{table}[htp]
\caption{Performance data for Example~\ref{ex:alg2} on $ [0,1.5]\times [-1.5,0] $.}
\begin{center}
\begin{tabular}{@ {}*{4}{ l}@ {}}\toprule
 & Step size $t\equiv 1$ & Step size $t\equiv 0.917 $ &  Adapt. $ \tau=0.1 $\\
 \midrule
 Average nr. of iterations & $16.7$ & $13.3$ & $6$  \\
 Average step size & $1$ & $0.917$ & $0.917$  \\
 \% of convergent iterations  & $81\%$ & $86\%$ & $97\%$ \\
 Average rate $\tilde{\rho}$ & $1.57$ & $1.11$ & $1.9$ \\
 \bottomrule
\end{tabular}
\end{center}
\label{performancetableexpsin}
\end{table}
\end{example}

\subsection{ODE Boundary Value Examples}
We shall now turn to ordinary boundary value problems.

\begin{example}\label{ex:ode1}
As a first example we discuss the nonlinear two-point boundary value problem given by

\begin{equation}
		\label{ex1}
	\left\{ \begin{aligned}
& u''+u^3=0, \ \text{on} \ (0,1), \\
& u(1)=u(0)=0.
\end{aligned} \right.
\end{equation}
Let us collect a few facts about \eqref{ex1}. Note that if $ u $ is a solution, then $ -u $ is as well a solution. Moreover, by a phase-plane analysis one can see that \eqref{ex1} has a unique positive solution $ u_+>0 $ (see~Figure~\ref{init} (right)). Thus, we have (at least) the three solutions $ \{u_0,u_+,u_-\} $ with $ u_-=-u_+ $, and $ u_0\equiv 0 $. Note that these solutions are roots of the nonlinear operator $ \F(u)=u''+u^3 $. Since, except for the trivial solution, we have no analytical solution formulas at hand, we will compare the numerical solutions and the corresponding exact solutions by means of their integral value over the domain $(0,1)$. Indeed, one can show (see, e.g., \cite{6}) that the unique positive solution $ u_{+} $ of \eqref{ex1} satisfies
\[
\int_{0}^{1}{u_{+}(x)\dx}=\frac{\pi}{\sqrt{2}}.
\]  
Consequently, we will identify the three solutions above with their corresponding integral values $ I_{S}=\left\{0,\nicefrac{\pi}{\sqrt{2}},-\nicefrac{\pi}{\sqrt{2}}\right\} $.

In our computations we determine numerical solutions of \eqref{ex1} by use of a standard finite element discretization based on piecewise linear basis functions (on uniform meshes with mesh size~$h=\nicefrac{1}{n}$, for some~$n\in\N$), and combine it with the Newton scheme~\eqref{damped}. Having computed such an approximate solution, we compare its integral value with the three values $I_{S}$ in order to decide to which solution our initial guess has converged. We will discuss this procedure in more detail in the sequel.

As initial guesses for the Newton iteration we use the following discrete set of piecewise linear continuous functions given by 
\begin{equation}
\label{initial}
u_{(i,j,0)}(0)=u_{(i,j,0)}(1)=0,\qquad
u_{(i,j,0)}(ih)=\alpha_j,
\end{equation} 
\begin{figure}
\includegraphics[width=0.45\textwidth]{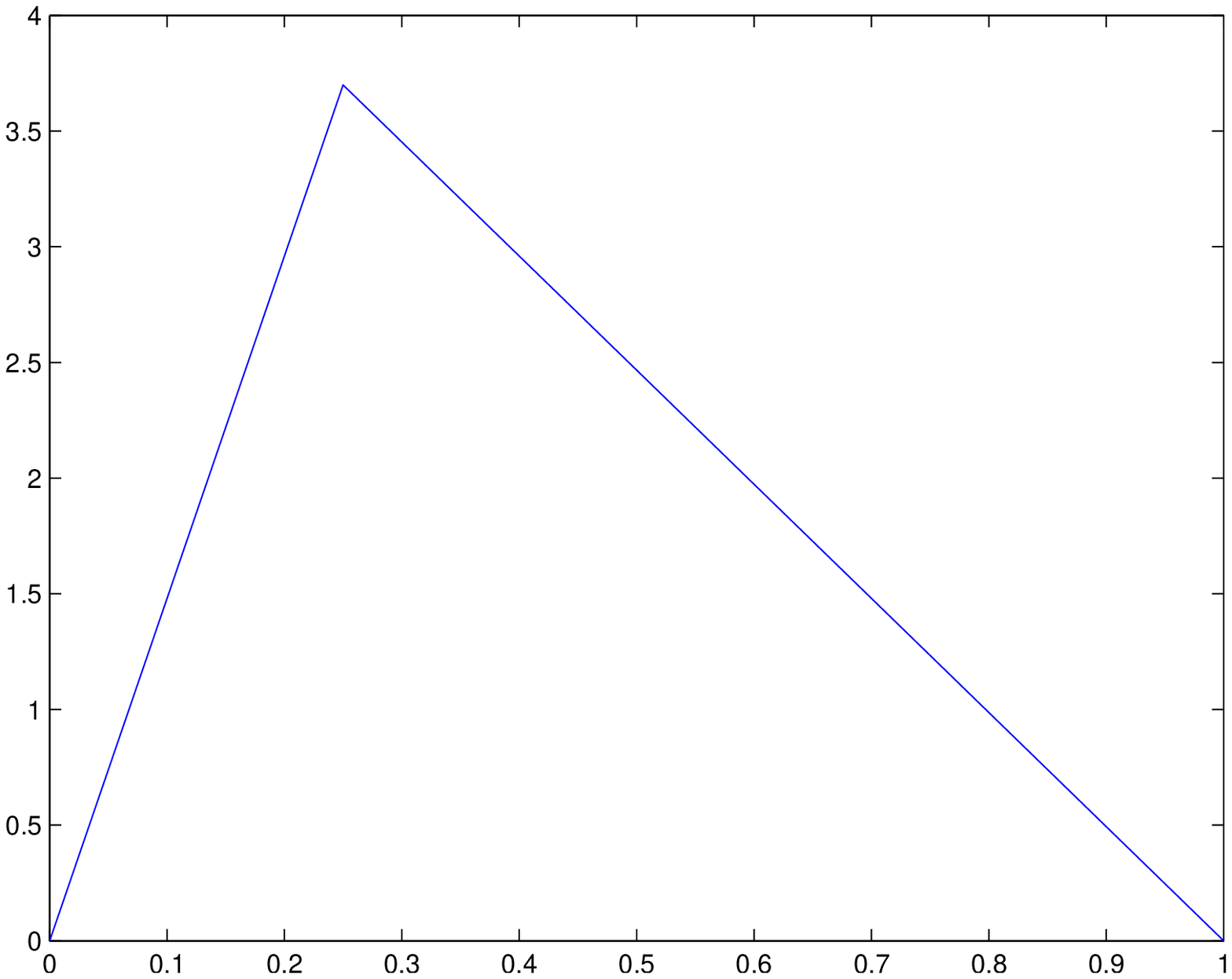}
\hfill
\includegraphics[width=0.45\textwidth]{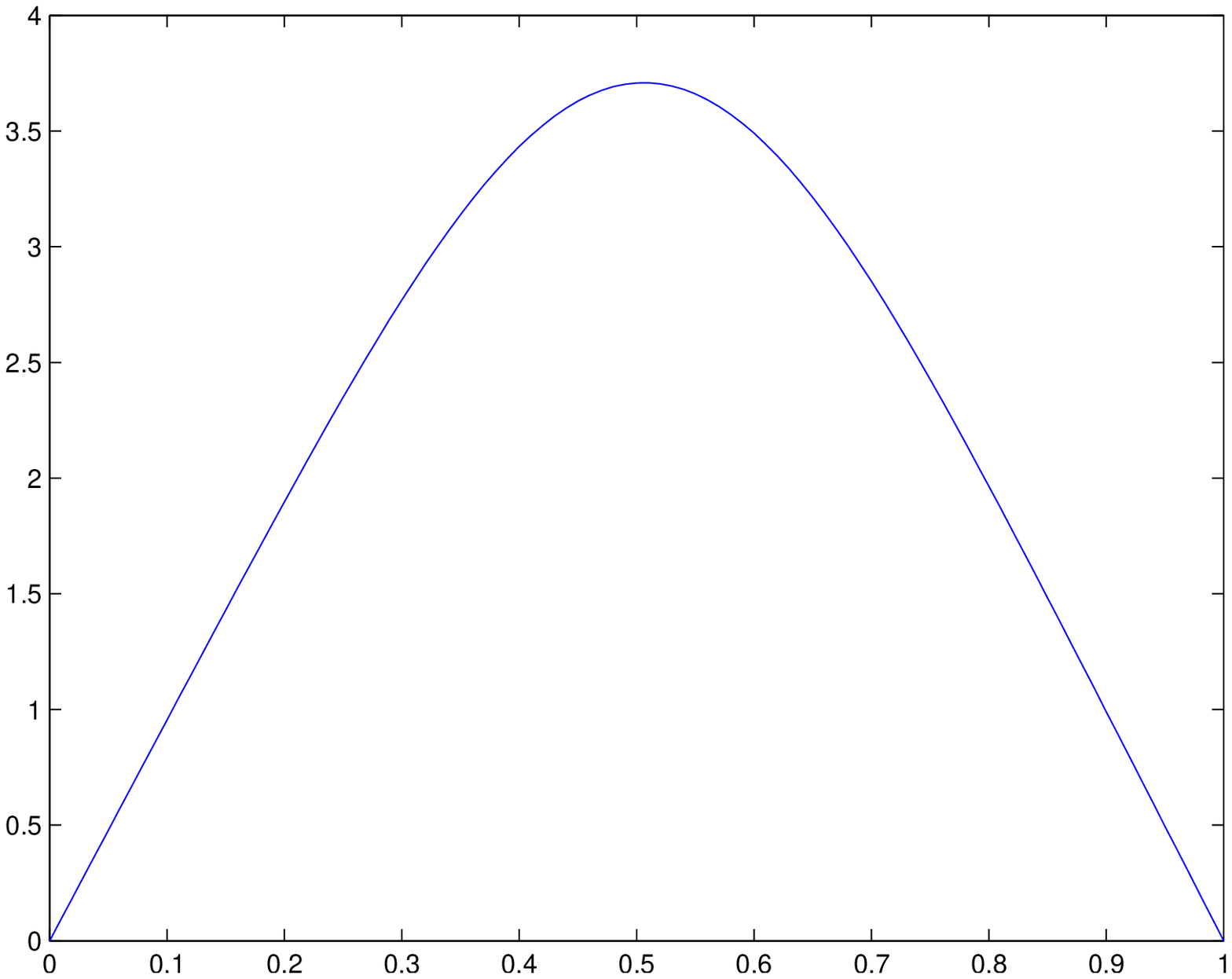}
\caption{Example of initial guess $u_{(i,j,0)}$ (left) and unique positive solution $ u_{+}$ of \eqref{ex1} (right).}
\label{init}
\end{figure}
for $ i \in \{1,2,\ldots,n-1\}$, and $\alpha_j \in [-4,4] $ with some range of indices for~$j$; cf.~Figure~\ref{init} (left) for an example. We can now visualize some finite dimensional subsets of the basins of attraction of the three solutions $\{u_0,u_+,u_-\}$ based on these initial guesses. More precisely, we identify an initial guess~$u_{i,j,0}$ given in \eqref{initial} by a point $ (ih,\alpha_j)$, where, for the computations, these points are taken from a uniform $ 400 \times 400 $ grid in the two-dimensional rectangle $(0,1)\times [-4,4]$. For each initial guess~$u_{i,j,0}$ we compute a sequence of solutions generated by the Newton method~\eqref{damped}, and determine the solution it converges to by checking the corresponding integral value from~$I_S$. The associated starting point $ (ih,\alpha_j)$ is then colored accordingly. This results in a two-dimensional plot showing a subset of the possibly infinite dimensional attractors of the three solutions.

It is reasonable to expect that the extremum value $ \alpha_j $ of the initial guess $ u_{(i,j,0)} $ will play an important role in the convergence behavior of the Newton scheme:
\begin{enumerate}
\item
For positive values $ \alpha_{j} $ close to the maximum of $ u_{+} $, we expect that the corresponding initial guess $ u_{(i,j,0)} $ converges to $ u_{+} $.    
\item
For negative values $ \alpha_{j} $ close to the minimum of $ u_{-} $, we expect that the corresponding initial guess $ u_{(i,j,0)} $ converges to $ u_{-} $.  
\item
For values $ \alpha_{j} $ close to $0 $, we expect that the corresponding initial guess $ u_{(i,j,0)} $ converges to the trivial solution $u_{0} $.  
\end{enumerate}

In Figure~\ref{odechaos} we present the three basins of attraction associated with the three solutions $\{u_+,u_-,u_0\}$  for both the traditional Newton-Galerkin scheme (with step size~1) and for the adaptive Newton-Galerkin method (Algorithm~\ref{algorithm}). 
\begin{figure}
\includegraphics[width=0.45\textwidth]{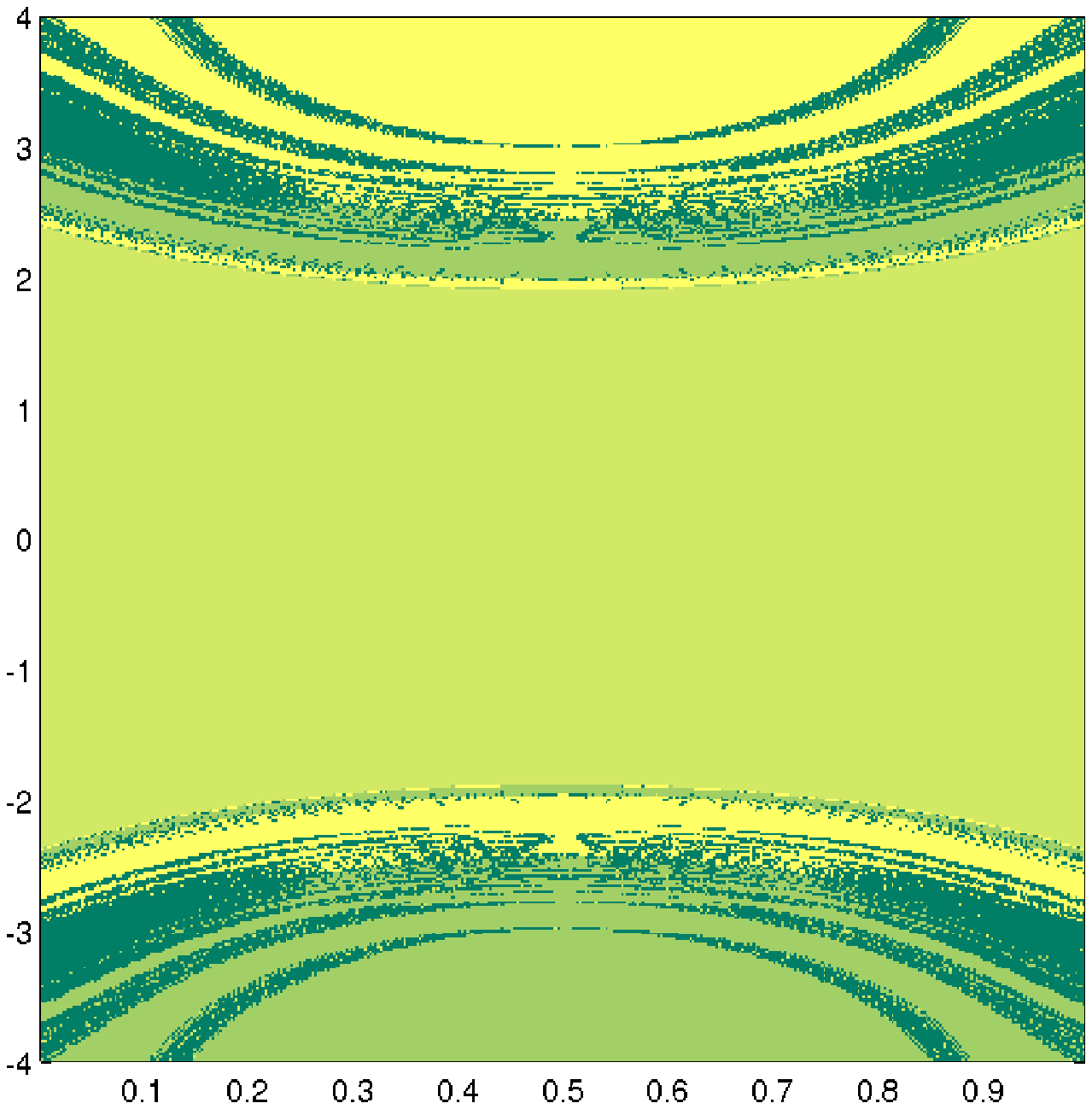}
\hfill
\includegraphics[width=0.45\textwidth]{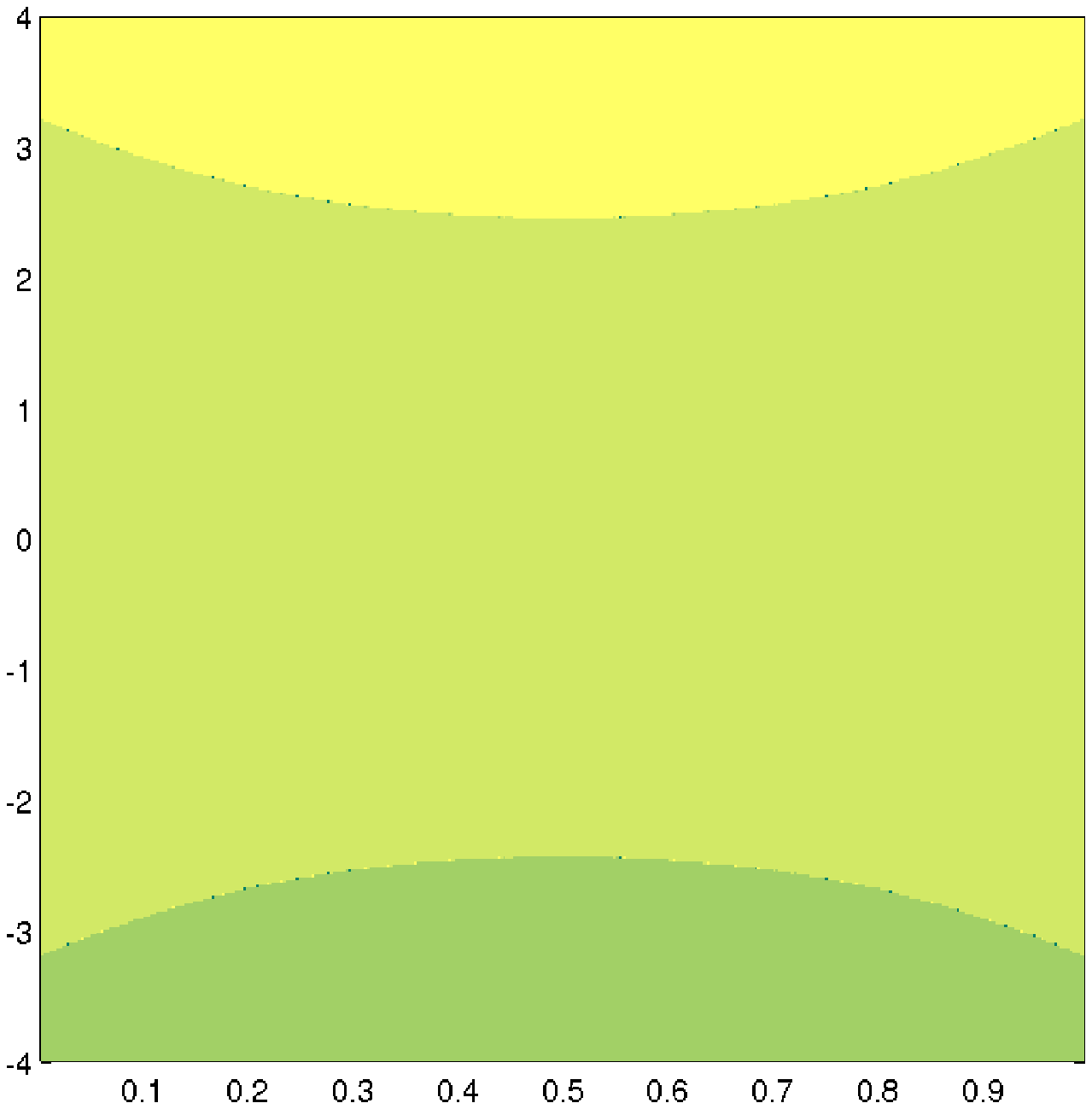}
\caption{The Newton-Galerkin method without (left) and with (right) step size control ($ \tau=0.1$).}
\label{odechaos}
\end{figure}
For the standard Newton-Galerkin method, we observe that there is a considerable number of  initial guesses which do not converge to the closest root (close in the sense of the average value of the exact solution). As in the algebraic example, moving an initial guess $ u_{(i,j,0)} $  to a sufficiently small neighborhood of a solution of \eqref{ex1} might not always be a well-conditioned procedure. Again, there are initial guesses which approach the area of quadratic convergence at a low rate or they visit various attractors before they approach a solution. The dark colored parts in Figure~\ref{odechaos} display the initial guesses $ u_{(i,j,0)} $ for which the iteration does not converge to one of the solutions~$\{u_0,u_+,u_-\}$ after a prescribed, maximal number of iterations. By applying step size control in the Newton iteration, we hope for more initial guesses $ u_{i,j,0} $ to converge, and moreover, for the chaotic behavior to be tamed considerably. This is indeed the case as becomes clear from Figure~\ref{odechaos} (right), where we clearly see that step size control in the case of solving ODEs by the Newton-Galerkin method is able to reproduce the boundaries between the attractors. 

\begin{figure}[htp]
\begin{center}
\includegraphics[width=0.75\textwidth]{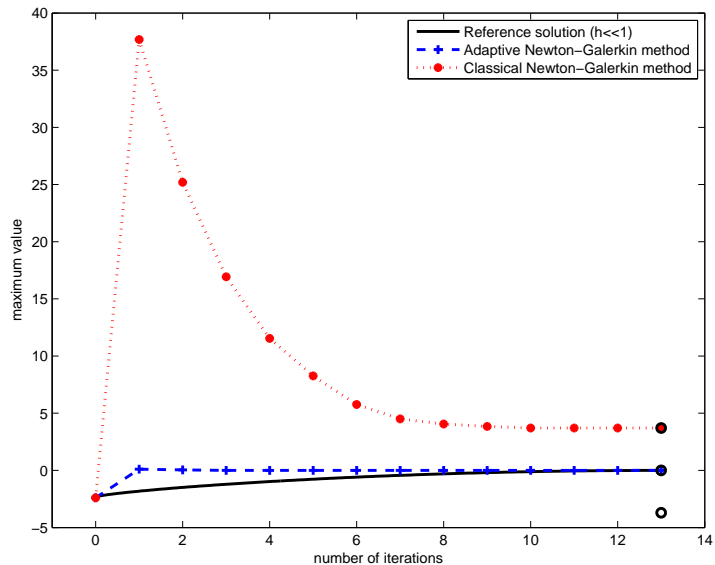}
\end{center}
\caption{Performance of the classical Newton-Galerkin and of the Newton-Galerkin method with adaptive step size control (with $\tau = 0.5 $) for the initial guess associated with the point~$ (0.5,-2.405) $. The vertical axis represents the extremal value of the corresponding iterate. The three small circles indicate the extremal values of the three solutions of \eqref{ex1}.}
\label{performanceplot3}
\end{figure}

In Figure~\ref{performanceplot3} we display the behavior of the classical Newton-Galerkin and of the adaptive Newton-Galerkin method with $ \tau=0.5 $, for the initial guess $ u_{i,j,0} $ with $n=100$, $ih=0.5$ and $ \alpha_{j}=-2.405 $, i.e., corresponding to the point~$(0.5,-2.405)$ which belongs to the attractor of~$u_0$. While the adaptive Newton method follows the exact trajectory closely and hence reaches the "correct" solution~$u_0$ of~\eqref{ex1}, we see that the classical Newton-Galerkin methods approaches the positive solution $ u_{+}$ instead. This is due a detour taken by the standard Newton method which is caused by an oversized update at the initial step. Also, notice that the adaptive scheme, as compared to the classical method, converges much faster to the associated zero.

In Table~\ref{perf22} we observe the benefits of step size control based on $10^4$ initial values of type \eqref{initial} with $ 	\alpha_{j} \in [-4,4] $. Again, an initial value $ u_{(i,j,0)} $ is considered convergent if it approaches the "correct solution" of \eqref{ex1}, i.e. the solution that is located in the same "exact" attractor as the initial value. The average numbers of iterations listed in Table~\ref{perf22} are determined such that, firstly, we obtain an absolute accuracy of at least $10^{-8}$ between the $n$-th and $(n+1)$-th iterates, and, secondly, the absolute error between the reference solution (which we computed with a small step size $t\ll1$) and the $(n+1)$-th iterate is at least $10^{-3}$. As before we compute an empirically determined convergence rate $ \tilde{\rho} $, where, incidentally, we only take into account those iterations which are convergent to the correct zero. The error in the $n$-th iteration is defined by
\[
e_{n}=\min{\abs{I_{S}-I_{N}(u_{i,j,n})}}.
\]
where $ I_{N}$ is the integral value of the numerical solution~$u_{i,j,n}$ resulting from~$n$ Newton steps for the initial value~$u_{i,j,0}$. We clearly observe a noticeable improvement in the average convergence rate $\tilde{\rho} $. Moreover almost all initial guesses converge, and the number of iterations is reduced by approximately $33\% $ compared to the traditional method.

\begin{table}[htp]
\caption{Performance data for Example~\ref{ex:ode1} for $ 10^4 $ initial guesses of type \eqref{initial}.}

\begin{tabular}{@ {}*{5}{ l}@ {}}\toprule
 & Step size $h\equiv 1$  &  Adaptive $ \tau=0.1 $   \\
 \midrule
 Average nr. of iterations & $23.5$ & $16$ \\
 Average step size & $1$ & $0.57$  \\
 \% of convergent iterations  & $74.5\%$  & $97\%$ \\
 Average rate $\tilde{\rho}$ & $1.4$ & $1.53$ \\
 \bottomrule
\end{tabular}
\label{perf22}
\end{table}
\end{example}

\begin{example}\label{ex:ode2}
As a second example, we consider the equation
\begin{equation}
		\label{bratu}
	\left\{ \begin{aligned}
& u'' + e^{u+1}=0 \ \text{on} \ (0,1), \\
& u(1)=u(0)=0,
\end{aligned} \right.
\end{equation}
which is also known as the $1$-D Bratu problem. We have the analytical solutions 
\[
u(x)=-2\ln\left(\frac{\cosh\left((x-\nicefrac{1}{2})\nicefrac{\theta}{2}\right)}{\cosh\left(\nicefrac{\theta}{4}\right)}\right),
\]
where $ \theta $ is determined by the transcendental equation
\begin{equation}
\label{theta} 
\theta = \sqrt{2 e}\cosh\left(\nicefrac{\theta}{4}\right).
\end{equation}
Note that there are exactly two solutions $\theta$ for \eqref{theta}, and hence, we have two solutions~$u_1$ and~$u_2$ of~\eqref{bratu} (see Figure~\ref{soli}).
\begin{figure}[htp]
\begin{center}
\includegraphics[width=0.5\textwidth]{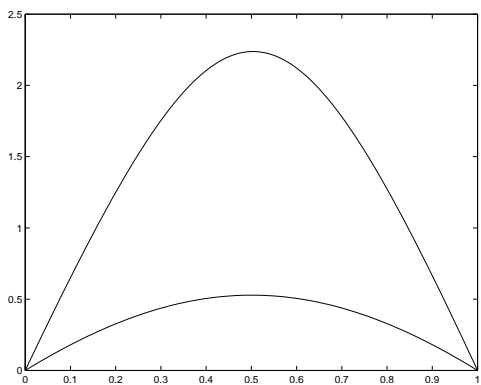}
\end{center}
\caption{The two exact solutions $u_1,u_2 $ of Bratu's equation \eqref{bratu}.}
\label{soli}
\end{figure}

As initial guesses for the Newton-Galerkin computations we again take the functions defined in \eqref{initial}, and compare the standard method with the one with step size control. In Figure~\ref{bratuclassic} we present the attractors for the traditional and the adaptive Newton-Galerkin methods by sampling $400\times 400 $ initial guesses corresponding to the points~$(ih,\alpha_j)$ in the rectangular domain  $ (0,1)\times [0,6] $. The yellow and green parts mark the attractors for the solution $ u_1 $ and $ u_2 $, respectively. We observe that, for the Newton iteration without step size control, there is a dark green shaded part separating the two domains of attraction. However, applying Algorithm~\ref{algorithm}, we observe that the boundaries of the different domains of attraction are nicely smoothed out. 

\begin{figure}
\includegraphics[width=0.45\textwidth]{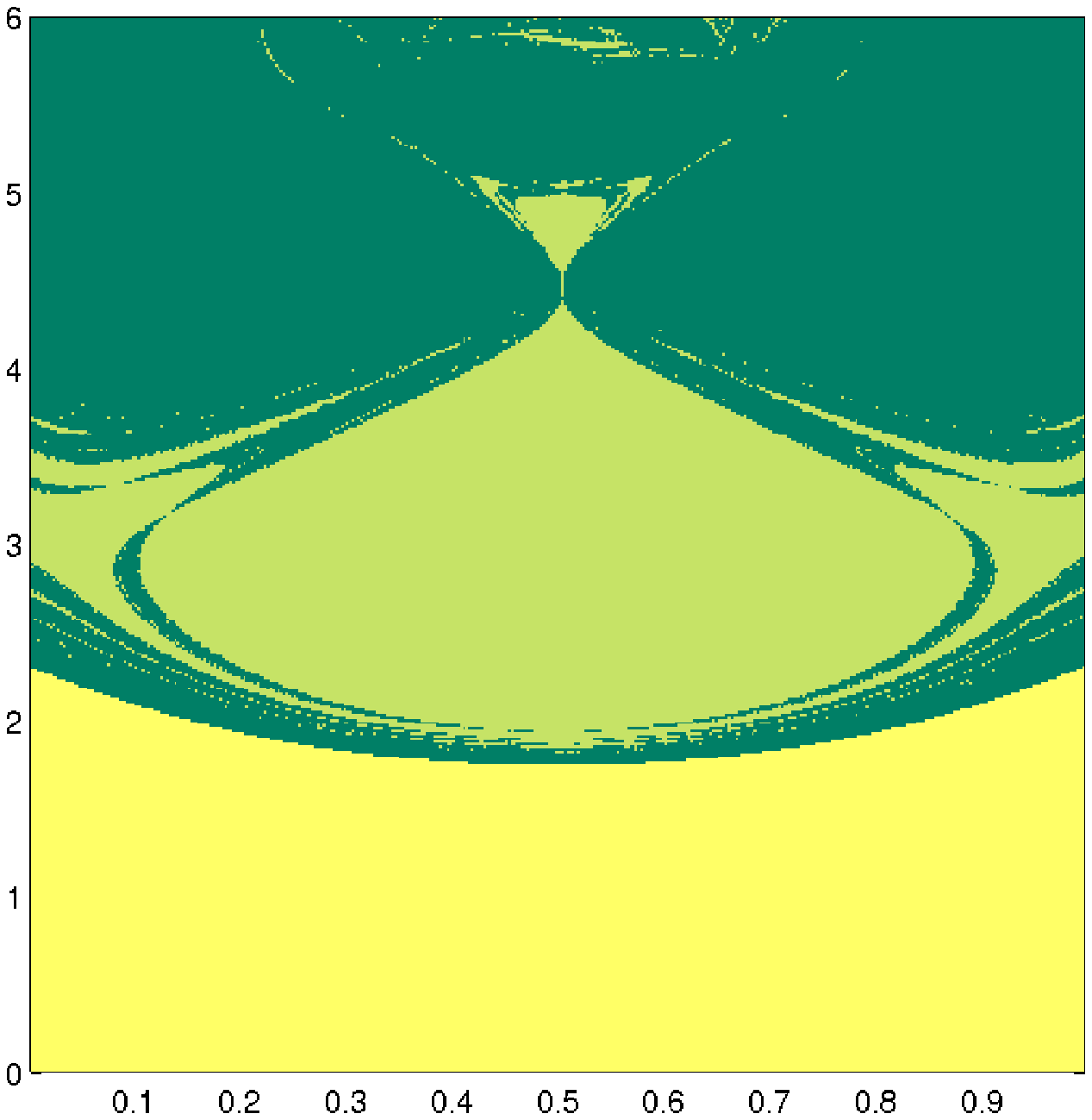}
\hfill
\includegraphics[width=0.45\textwidth]{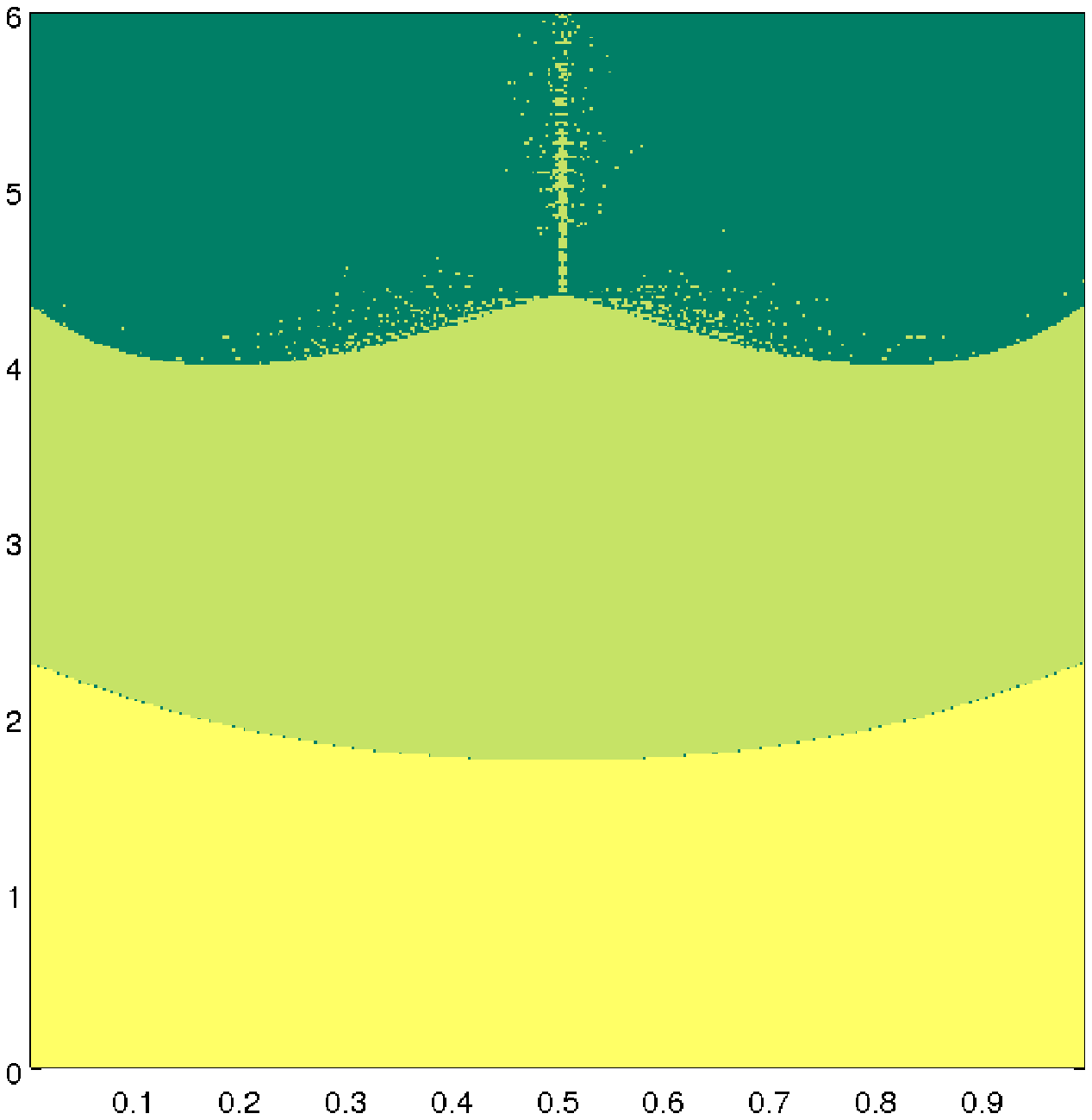}
\caption{The Newton-Galerkin method without (left) and with (right) step size control ($\tau = 0.1$).}
\label{bratuclassic}
\end{figure}

Table~\ref{perf33} is based on the information of $10^4$ initial guesses of type \eqref{initial} with $ \alpha_{j} \in [0,3] $. Note that the larger average iteration number in the adaptive approach comes from the fact that the classical Newton-Galerkin method breaks down for initial guesses within the dark green shaded part and therefore does not reach the maximal number of iterations. However, note that, employing a step size control procedure, increases the number of convergent initial guesses $u_{(i,0)}$ remarkably.

\begin{table}[htp]
\caption{Performance data for Example~\ref{ex:ode2} for $  10^4 $ initial guesses of type \eqref{initial}.}
\begin{center}
\begin{tabular}{@ {}*{5}{ l}@ {}}\toprule
 & $\text{Step size}\equiv 1$  &  $ \text{Adaptive}  \ \tau = 0.1 $  \\
 \midrule
 Average nr. of iterations & $10$ & $14$ \\
 Average step size & $1$ & $0.625$  \\
 \% of convergent iterations  & $83.5\%$ & $98.5\%$ \\
 Average rate $ \tilde{\rho}$ & $1.9$ & $1.2$ \\
 \bottomrule
\end{tabular}
\end{center}
\label{perf33}
\end{table} 
\end{example}

\subsection{A PDE Boundary Value Example} We close this application section with a partial differential equation example.

\begin{example}\label{ex:pde} Consider the boundary value problem
\begin{equation}
\label{pde}
\left\{ \begin{aligned}
& \Delta u + u^3=0 \ \text{in} \ \Omega, \\
& u=0, \ \text{on} \ \partial \Omega,
\end{aligned} \right.
\end{equation}
where $ \Omega=[0,1]^{2} $ is the unit square in~$\mathbb{R}^2$. Again, we are interested in three particular solutions $ \left\{u_{0},u_{+},u_{-}\right\} $, which are globally zero, positive, and negative on~$\Omega$, respectively.
\begin{figure}
\includegraphics[width=0.45\textwidth]{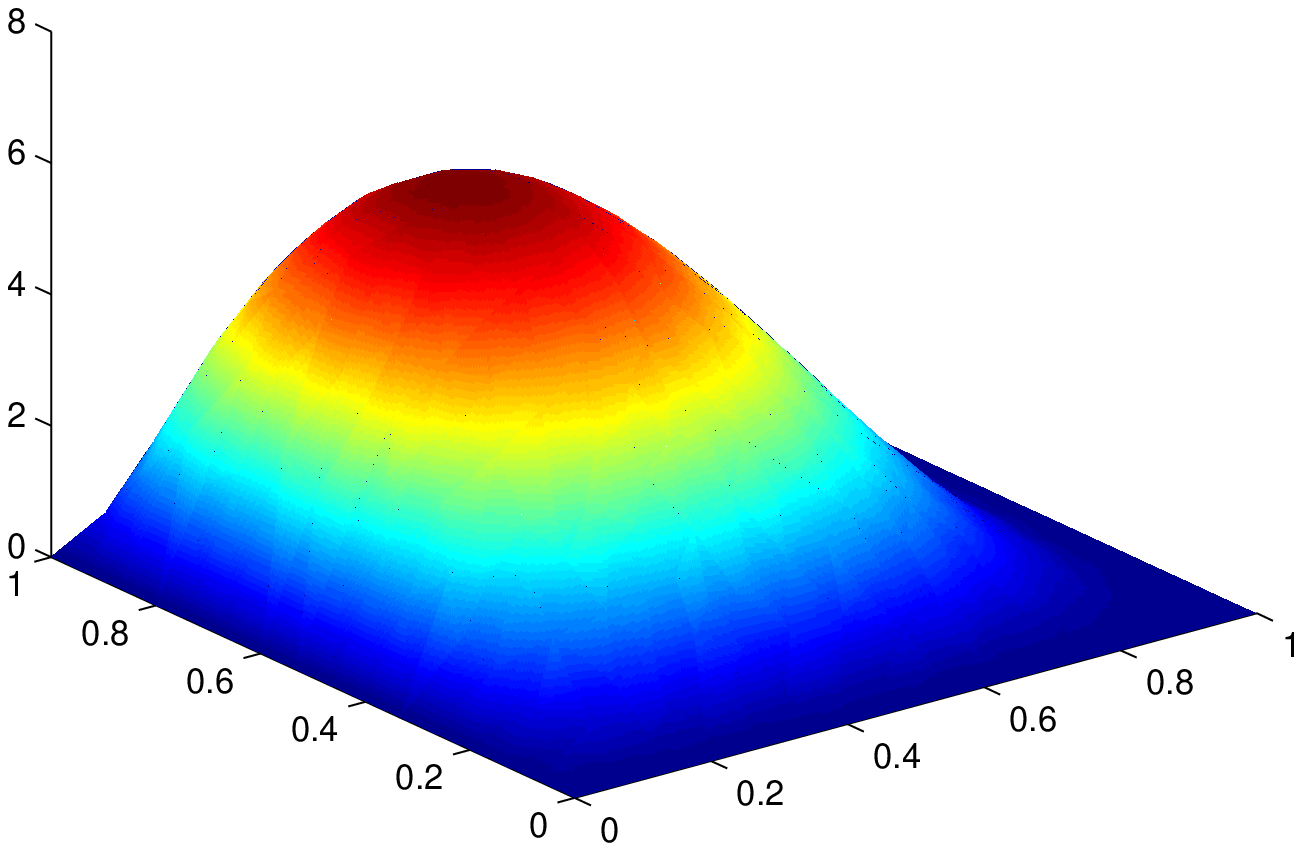}
\hfill
\includegraphics[width=0.45\textwidth]{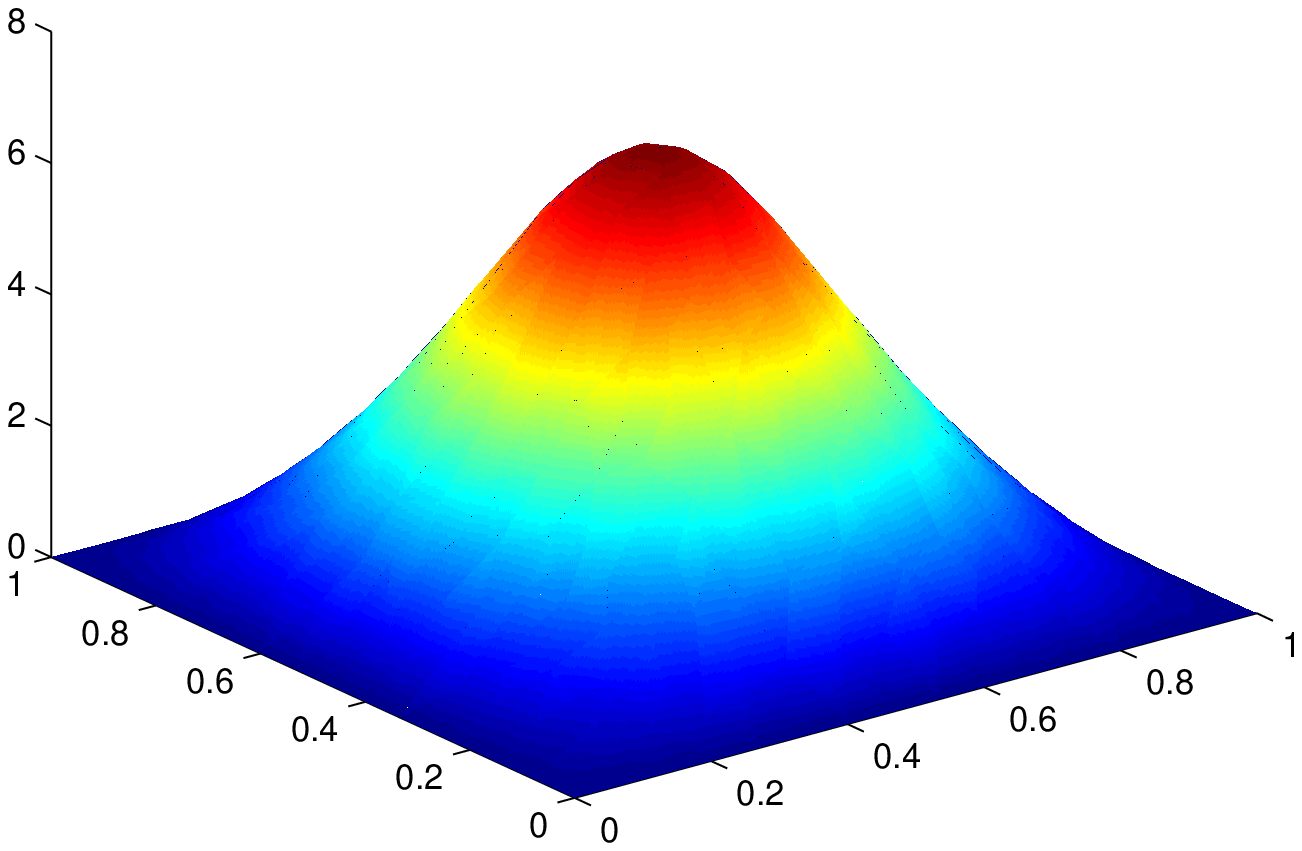}
\caption{Example of an initial guess from~\eqref{initpde} (left), and unique positive solution~$u_+$ of \eqref{pde} (right).}
\label{reference}
\end{figure}

Consider the hill-shaped functions
\begin{equation}
\phi_{(k,j,n)}(x,y)=\left(\frac{x}{x+\varepsilon}\right)^{k}\left(\frac{y}{y+\varepsilon}\right)^{j}\left(\frac{1-x}{1-x+\varepsilon}\right)^{n-k}\left(\frac{1-y}{1-y+\varepsilon}\right)^{n-j},
\end{equation}
with~$\varepsilon=\nicefrac1n$. Then, define the initial guesses for the Newton-Galerkin iteration (see Figure~\ref{reference}) as follows: For a fixed $n\in \N $, and $ k,j\in\{1,\ldots, n-1\} $, $ i \in \{-c,-c+\frac{1}{n},\ldots,c-\frac{1}{n},c\} $, with $ c\in\R $, we set 
\begin{equation}
\label{initpde}
u_{i,k,j,n}=\frac{i}{\norm{\phi_{(k,j,n)}}_{L^{\infty}(\Omega)}} \phi_{(k,j,n)}.
\end{equation}  

%
%
%
%

In Figure~\ref{Pdefractal} we show (finite dimensional subsets of) the attractors of the Newton-Galerkin method without step size control by sampling $10^6$ initial guesses (for $c=8$). As in the ODE case the dark-green shaded parts indicate the initial values which are not convergent to any of the three solutions of Figure~\ref{Pdefractal}. For the sake of clarity, we extract three horizontal slices from these plots, namely the one in the middle, at a quarter and on top of the cubes and display them in Figure~\ref{Pdefractal} with resolution $ 500\times 500 $. One can clearly see the chaotic behavior of the classical Newton-Galerkin method; indeed, there are again a large number of initial guesses which do not converge to the closest solution (as in the ODE-case we call an approximate solution close to the exact solution of \eqref{pde} if it is close in the mean, that is, in the integral sense). In addition, we present the basins of attraction based on step size control with~$\tau=0.1 $. As in the previous examples step size control is able to tame the chaotic behavior of the classical Newton method. Moreover, the boundaries of the three different basins of attraction are resolved, and the domain of attraction for the three solutions under consideration is considerably enlarged in the given range. 

\begin{figure}
\includegraphics[width=0.45\textwidth]{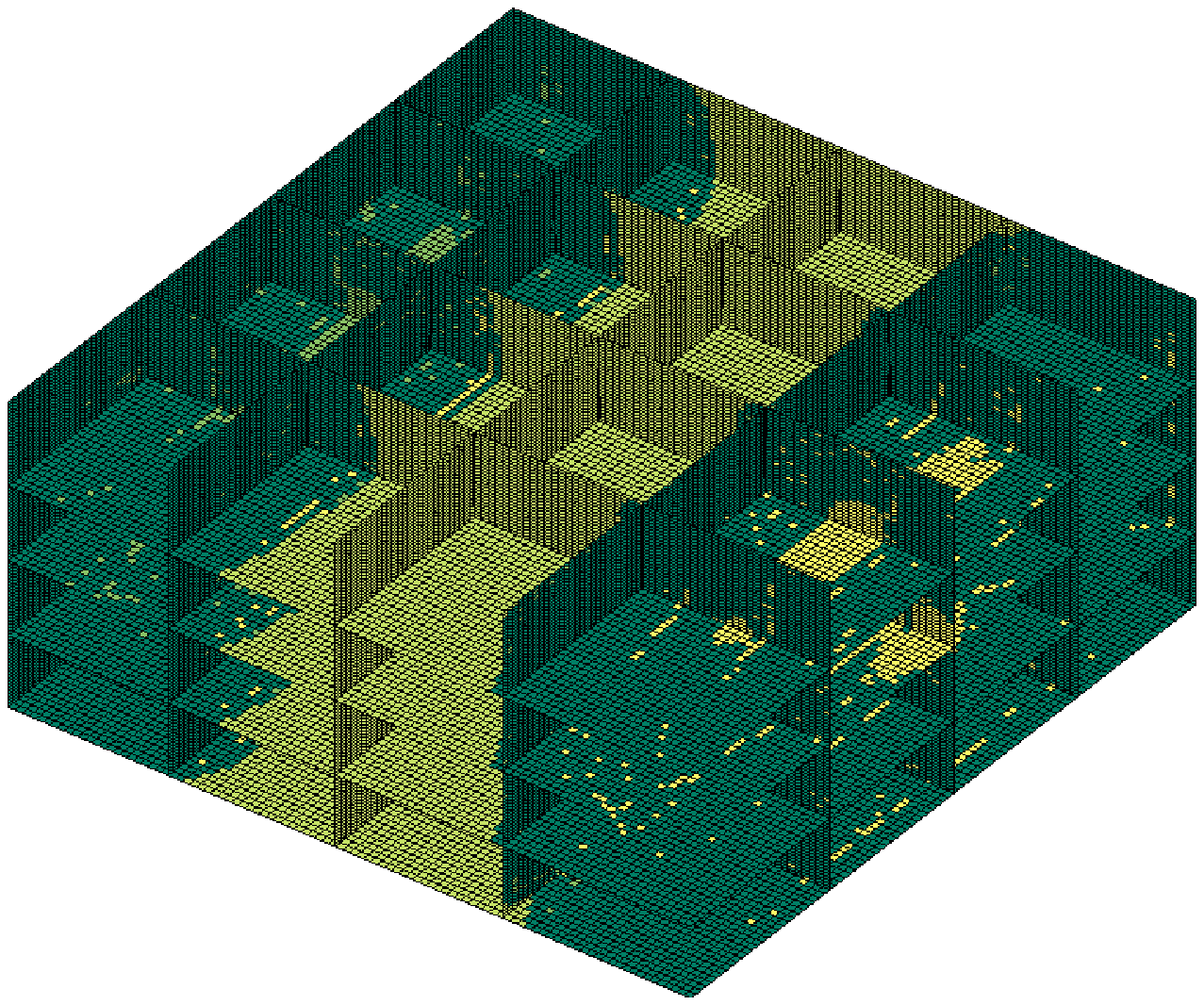}
\hfill
\includegraphics[width=0.45\textwidth]{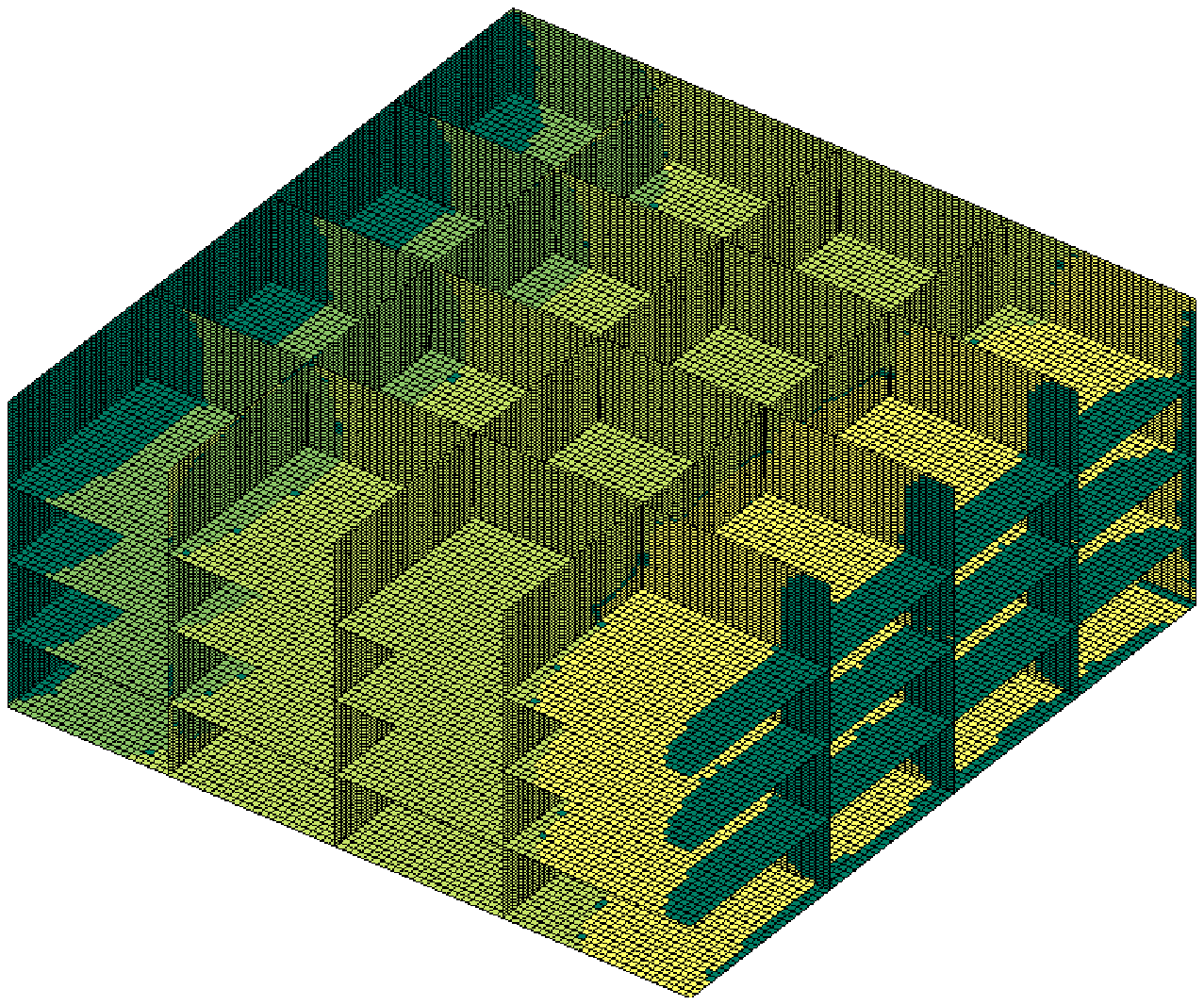}
\caption{Newton-Galerkin method with (left) and without (right) step size control (with~$\tau=0.1$).}
\label{Pdefractal}
\end{figure}

\begin{figure}
\includegraphics[width=0.45\textwidth]{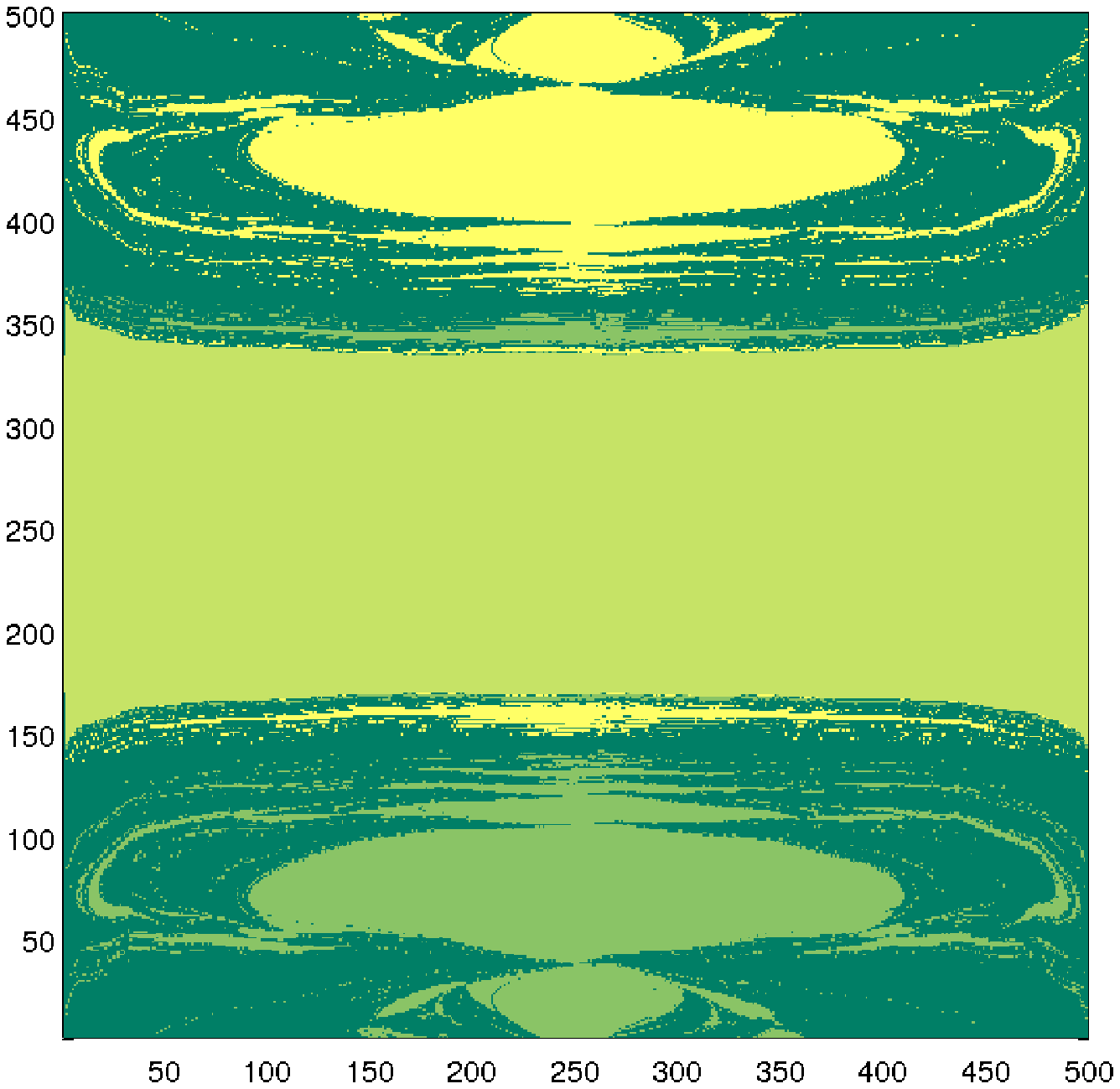}
\hfill
\includegraphics[width=0.45\textwidth]{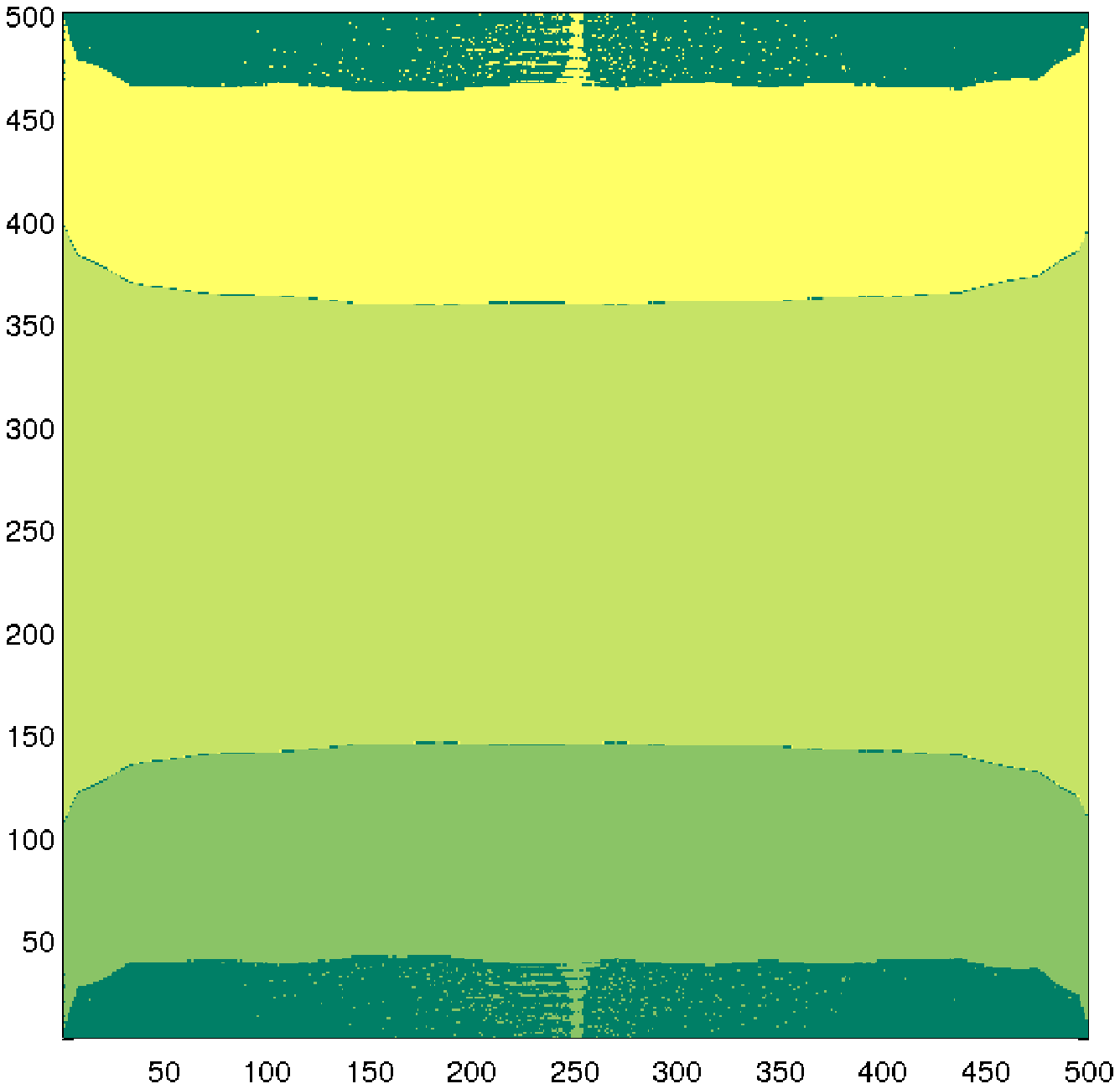}
\includegraphics[width=0.45\textwidth]{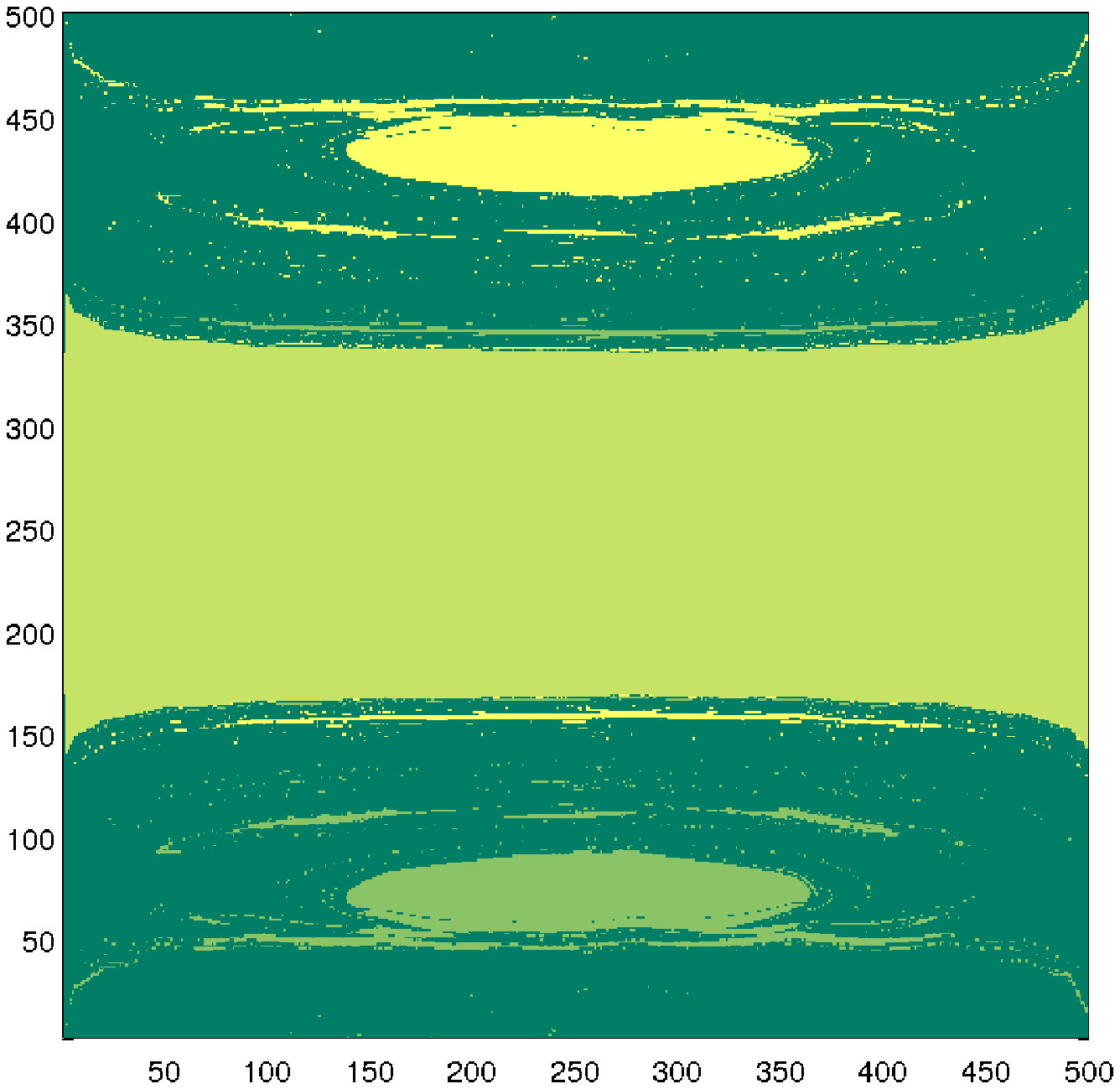}
\hfill
\includegraphics[width=0.45\textwidth]{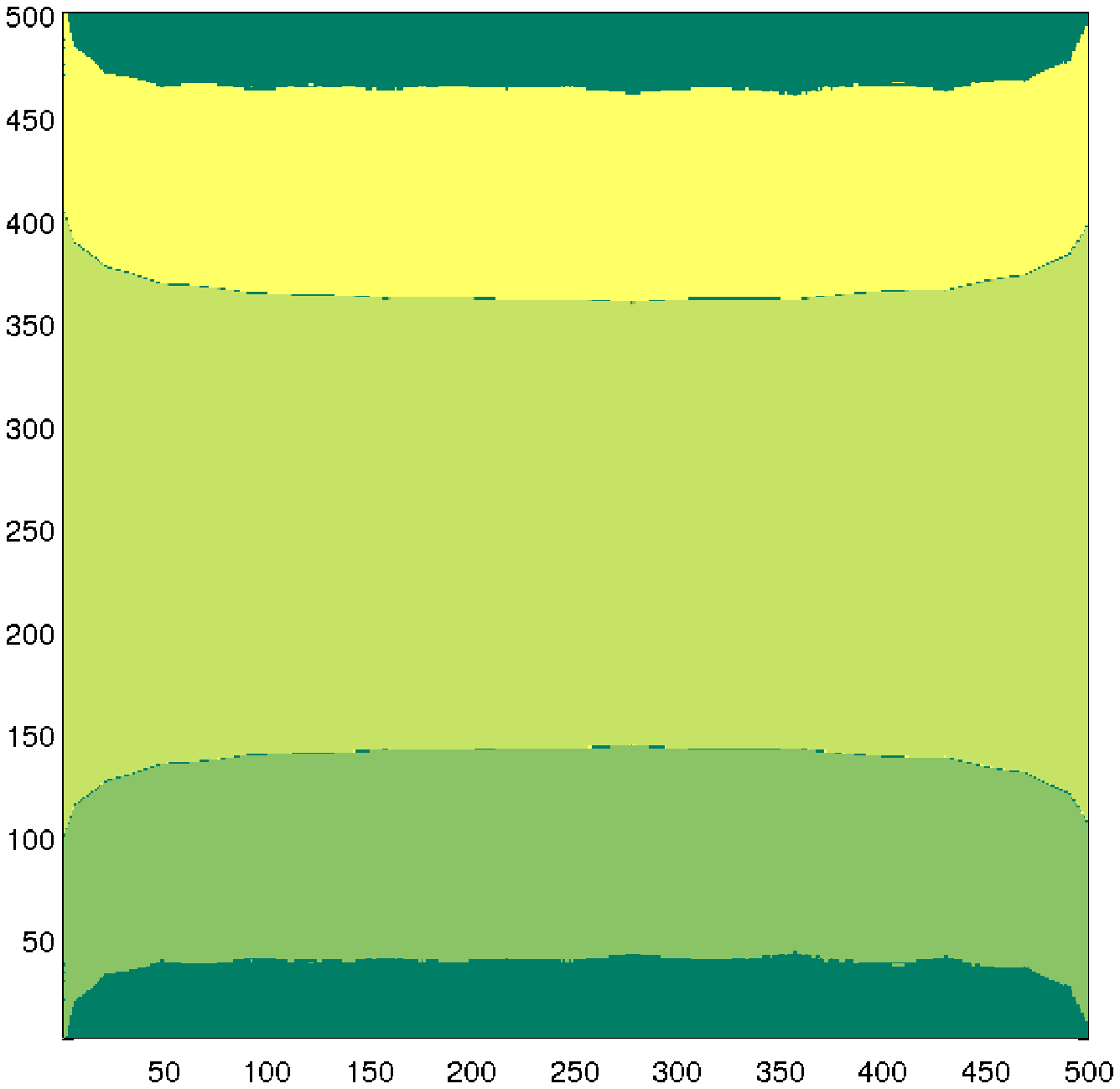}
\includegraphics[width=0.45\textwidth]{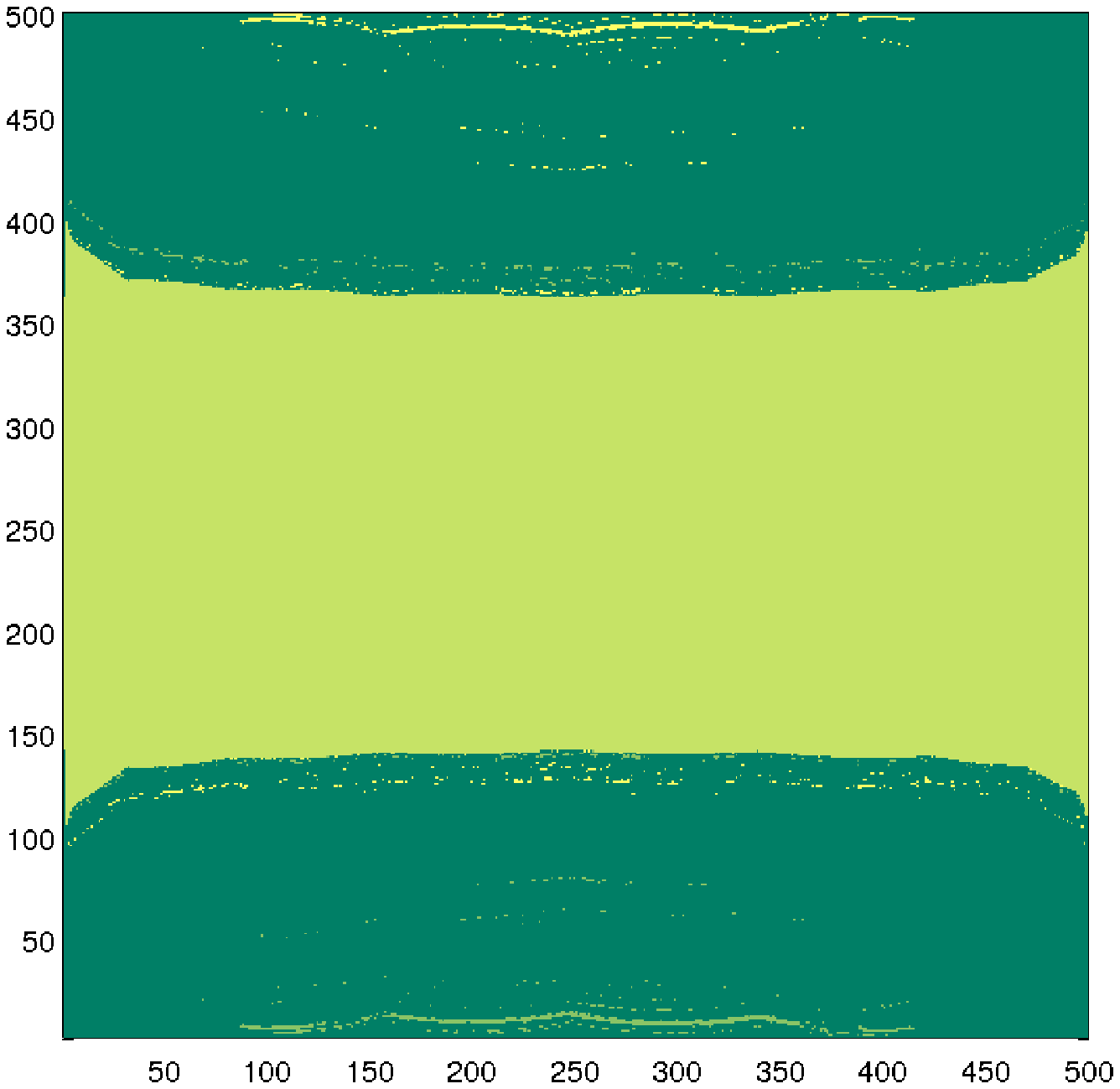}
\hfill
\includegraphics[width=0.45\textwidth]{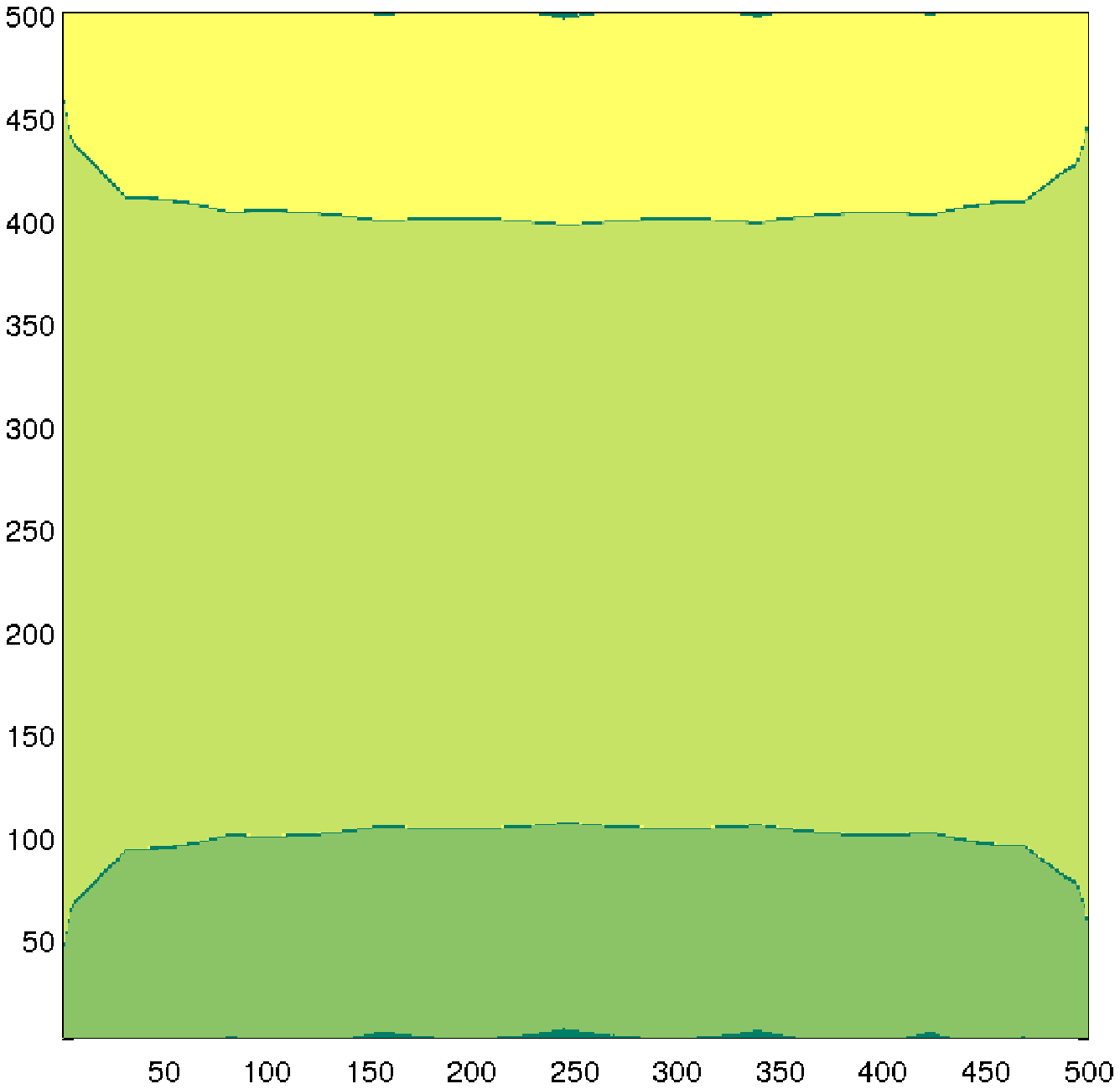}
\caption{Three slices without (left) and with (right) step size control (with~$\tau =0.1 $).}
\label{slices}
\end{figure}

\end{example}

\section{Conclusions}
In this paper we have introduced an adaptive Newton method for (nonlinear) operator equations, $\F(x)=0$, in Banach spaces. While adaptive Newton methods are popular instruments in the area of numerical optimization, our approach makes use of the dynamical system character of the continuous Newton method, $\dot{x}=\NF\F(x)$. Indeed, this system can be seen as a preconditioned version of the system~$\dot{x}=\F(x)$ by~$(\F')^{-1}$. It has, on the one hand, the very favorable property of all zeros being attractive, on the other hand, however, singularities in~$\F'$ may cause the associated discrete system to exhibit chaotic behavior. In order to tame the chaos of the discrete Newton flow, we have proposed a simple, prediction-type, adaptive step size control procedure whose purpose is to follow the flow of the continuous system to a reasonable extent, i.e., in particular, by avoiding to switch between different attractors. We have tested our method in the context of algebraic systems and of finite element discretizations for boundary value problems. The goal of our experiments was to demonstrate empirically that the proposed scheme is indeed capable of taming the chaotic regime of the traditional Newton-Raphson method, at least in the available setting of two-dimensional graphical representations. Our experiments strongly indicate that the adaptive method in this paper performs very well for the examples considered here: in particular, the graphics reveal that fractal attractor boundaries are being smoothed out, high convergence rates can be retained, and the domains of convergence can be enlarged. Our future research will focus on the combination of the proposed approach with adaptive discretization methods for high- or even infinite-dimensional problems.

\bibliographystyle{plain}
\bibliography{references}

\end{document}